\documentclass[12pt, reqno]{amsart} % AMS article style
\usepackage{amssymb,amscd,amsfonts,amsbsy}
\usepackage{latexsym}
\usepackage{exscale}
\usepackage{amsmath,amsthm,amsfonts}
\usepackage{mathrsfs}
\usepackage{xcolor} % Enables the creations of colors.
\usepackage[colorlinks=true,linkcolor=blue,citecolor=red,urlcolor=red, backref=page]{hyperref} % Add coloured links
\usepackage{esint} % For \fint symbol
\usepackage{stmaryrd}
\usepackage{pifont}

\usepackage[utf8]{inputenc}

\calclayout

\allowdisplaybreaks%[3] % Allow equation break between pages

\newtheorem{theorem}{Theorem}[section]

\newtheorem{lemma}[theorem]{Lemma}

\newtheorem{definition}[theorem]{Definition}

\numberwithin{equation}{section} %numbers equations within sections.

\def\A{\mathcal{A}}

\def\D{\mathcal{D}}

\def\S{\mathcal{S}}

\def\Z{\mathbb{Z}}
\def\R{\mathbb{R}}
\def\Rn{\mathbb{R}^n}

\def\BMO{\operatorname{BMO}}

\def\loc{\operatorname{loc}}
\def\supp{\operatorname{supp}}

\let \o=w

\begin{document}

\title[Square functions associated with operators]
  {Weak and strong types estimates for square functions associated with operators}

\author{Mingming Cao}
\address{Mingming Cao\\
Instituto de Ciencias Matem\'aticas CSIC-UAM-UC3M-UCM\\
Con\-se\-jo Superior de Investigaciones Cient{\'\i}ficas\\
C/ Nicol\'as Cabrera, 13-15\\
E-28049 Ma\-drid, Spain} \email{mingming.cao@icmat.es}

\author{Zengyan Si}
\address{Zengyan Si\\
School of Mathematics and Information Science\\
Henan Polytechnic University\\
Jiaozuo 454000\\
People's Republic of China} \email{zengyan@hpu.edu.cn}

\author{Juan Zhang}
\address{Juan Zhang\\
School of Science\\
Beijing Forestry University\\
Beijing, 100083 \\
People's Republic of China}\email{juanzhang@bjfu.edu.cn}

\thanks{The first author acknowledges financial support from the Spanish Ministry of Science and Innovation, through the ``Severo Ochoa Programme for Centres of Excellence in R\&D'' (SEV-2015-0554) and from the Spanish National Research Council, through the ``Ayuda extraordinaria a Centros de Excelencia Severo Ochoa'' (20205CEX001).The second author was sponsored by  Natural Science Foundation of Henan(No.202300410184), the Key Research Project for Higher Education in Henan Province(No.19A110017) and the Fundamental Research Funds for the Universities of Henan Province(No.NSFRF200329). The third author was supported by the Fundamental Research Funds for the Central Universities (No.BLX201926). }

\subjclass[2010]{42B20, 42B25}
%42B20 Singular and oscillatory integrals (Calderon-Zygmund, etc.)
%42B25 Maximal functions, Littlewood-Paley theory

\keywords{Square functions,
Bump conjectures,
Mixed weak type estimates,
Local decay estimates}

\date{November 19, 2020}

%%%%%%%%%%%%%%%%%%%% ABSTRACT ABSTRACT ABSTRACT %%%%%%%%%%%%%%%%%%%%%%
\begin{abstract}
Let $L$ be a linear operator in $L^2(\Rn)$  which generates a semigroup $e^{-tL}$ whose kernels $p_t(x,y)$ satisfy the
Gaussian upper bound. In this paper, we investigate several kinds of weighted norm inequalities for the conical square function $S_{\alpha,L}$ associated with an abstract operator $L$. We first establish two-weight inequalities including bump  estimates, and Fefferman-Stein inequalities with arbitrary weights. We also present the local decay estimates using the extrapolation techniques, and the mixed weak type estimates corresponding Sawyer's conjecture by means of a Coifman-Fefferman inequality. Beyond that, we consider other weak type estimates including the restricted weak-type $(p, p)$ for $S_{\alpha, L}$ and the endpoint estimate for commutators of $S_{\alpha, L}$. Finally, all the conclusions aforementioned can be applied to a number of square functions associated to $L$.
\end{abstract}
%%%%%%%%%%%%%%%%%%%% ABSTRACT ABSTRACT ABSTRACT %%%%%%%%%%%%%%%%%%%%%%

\maketitle
%\tableofcontents

%%%%%%%%%%%%%%%%%%%%%%% SECTION SECTION SECTION %%%%%%%%%%%%%%%%%%%%%%%
%%%%%%%%%%%%%%%%%%%%%%% SECTION SECTION SECTION %%%%%%%%%%%%%%%%%%%%%%%
\section{Introduction}\label{Introduction}

Given an operator $L$, the conical square function $S_{\alpha,L}$ associated with $L$ is defined by 
\begin{align}\label{def:SaL}
S_{\alpha,L}(f)(x) :=\bigg(\iint_{\Gamma_{\alpha}(x)}|t^mLe^{-t^mL}f(y)|^2\frac{dydt}{t^{n+1}}\bigg)^{\frac12}, 
\end{align}
where $\Gamma_{\alpha}(x)=\{(x,t)\in \Rn \times (0,\infty):|x-y|<\alpha t\}$. In particular, if $m=2$ and $L=-\Delta$, $S_{\alpha,L}$ is the classical area integral function. The conical square functions associated with abstract operators played an important role in harmonic analysis. For example, by means of $S_{\alpha, L}$, Auscher et al. \cite{ADM} introduced the Hardy space $H^1_L$ associated with an operator $L$. Soon after, Duong and Yan \cite{DY} showed that $\BMO_{L^*}$ is the dual space of the Hardy space $H^1_L$, which can be seen a generalization of Fefferman and Stein's result on the duality between  $H^1$ and $\BMO$ spaces. Later, the theory of function spaces associated with operators has been developed and generalized to many other different settings, see for example \cite{DL, HM, HLMMY, LW}. Recently, Martell and Prisuelos-Arribas \cite{MP-1} studied the weighted norm inequalities for conical square functions. More specifically, they established boundedness and comparability in weighted Lebesgue spaces of different square functions using the Heat and Poisson semigroups. Using these square functions, they \cite{MP-2} define several weighted Hardy spaces $H_L^1(w)$ and showed that they are one and the same in view of the fact that the square functions are comparable in the corresponding weighted spaces. Very recently, Bui and Duong \cite{BD} introduced several types of square functions associated with operators and established the sharp weighted estimates.  

In this paper, we continue to investigate several kinds of weighted norm inequalities for such operators, including bump  estimates, Fefferman-Stein inequalities with arbitrary weights, the local decay estimates, the mixed weak type estimates corresponding Sawyer's conjecture. Beyond that, we consider other weak type estimates including the restricted weak-type $(p, p)$ estimates and the endpoint estimate for the corresponding commutators. For more information about the progress of these estimates, see \cite{CXY, F2, OPR, PW, MW, S83} and the reference therein.

Suppose that $L$ is  an operator which satisfies the following properties:
\begin{enumerate}
\item[(A1)] $L$ is a closed densely defined operators of type $\omega$ in $L^2(\mathbb{R}^n)$ with $0\leq \omega< \pi/ 2$, and it has a bounded $H_\infty$-functional calculus in $L^2(\mathbb{R}^n)$.
\vspace{0.2cm}
\item[(A2)] The kernel $p_t(x,y)$ of $e^{-tL}$ admits a Gaussian upper bound. That is, there exists $m\geq 1$ and $C,c>0$ so that for all $x,y\in \mathbb{R}^n$ and $t>0,$
$$|p_t(x,y)|\leq \frac{C}{t^{n/ m}}\exp\bigg(-\frac{|x-y|^{m/(m-1)}}{c \, t^{1/(m-1)}}\bigg).$$
\end{enumerate}

Examples of the operator $L$ which satisfies condition (A1) and (A2) include: Laplacian $-\Delta$ on $\mathbb{R}^n$, or the Laplace operator on an open connected domain
with Dirichlet boundary conditions, or the homogeneous sub-Laplacian on a homogeneous group; Schr\"{o}dinger operator $L=-\Delta+V$  with a nonnegative potential $0\leq V \in L^1_{\loc}(\Rn)$.

The main results of this paper can be stated as follows. We begin with the bump estimates for $S_{\alpha, L}$.
%%%%%%%%%%%%%%%%%%%%%%% THEOREM THEOREM THEOREM %%%%%%%%%%%%%%%%%%%%%
\begin{theorem}\label{thm:Suv}
Let $1<p<\infty$, and let $S_{\alpha, L}$ be defined in \eqref{def:SaL} with $\alpha \ge 1$ and $L$ satisfying {\rm (A1)} and {\rm (A2)}. Given Young functions $A$ and $B$, we denote
\begin{equation*}
\|(u, v)\|_{A, B, p} :=
\begin{cases}
\sup\limits_{Q} \|u^{\frac1p}\|_{p, Q} \|v^{-\frac1p}\|_{B, Q}, & \text{if } 1<p \le 2,
\\%%%%%%%%%%%%%
\sup\limits_{Q} \|u^{\frac2p}\|_{A, Q}^{\frac12} \|v^{-\frac1p}\|_{B,Q}, & \text{if } 2<p<\infty.
\end{cases}
\end{equation*}
If the pair $(u, v)$ satisfies $||(u, v)||_{A, B, p}<\infty$ with $\bar{A} \in B_{(p/2)'}$ and $\bar{B} \in B_p$, then
\begin{align}\label{eq:SLp}
\|S_{\alpha,L}(f)\|_{L^p(u)} &\lesssim \alpha^{n} \mathscr{N}_p \|f\|_{L^p(v)},
\end{align}
where
\begin{equation*}
\mathscr{N}_p  :=
\begin{cases}
||(u, v)||_{A,B,p} [\bar{B}]_{B_p}^{\frac1p}, & \text{if } 1<p \le 2,
\\%%%%%%%%%%%%%
||(u, v)||_{A,B,p} [\bar{A}]_{B_{(p/2)'}}^{\frac12-\frac1p} [\bar{B}]_{B_p}^{\frac1p}, & \text{if } 2<p<\infty.
\end{cases}
\end{equation*}
\end{theorem}
%%%%%%%%%%%%%%%%%%%%%%% THEOREM THEOREM THEOREM %%%%%%%%%%%%%%%%%%%%%

%%%%%%%%%%%%%%%%%%%%%%% THEOREM THEOREM THEOREM %%%%%%%%%%%%%%%%%%%%%
\begin{theorem}\label{thm:Sweak}
Let $1<p<\infty$, and let $S_{\alpha, L}$ be defined in \eqref{def:SaL} with $\alpha \ge 1$ and $L$ satisfying {\rm (A1)} and {\rm (A2)}. Let $A$ be a Young function. If the pair $(u, v)$ satisfies $[u, v]_{A,p'}<\infty$ with $\bar{A} \in B_{p'}$, then
\begin{align}\label{eq:S-weak}
\|S_{\alpha,L}(f)\|_{L^{p,\infty}(u)} \lesssim [u, v]_{A,p'} [\bar{A}]_{B_{p'}}^{\frac{1}{p'}} \|f\|_{L^p(v)}.
\end{align}
\end{theorem}
%%%%%%%%%%%%%%%%%%%%%%% THEOREM THEOREM THEOREM %%%%%%%%%%%%%%%%%%%%%

We next present the Fefferman-Stein inequalities with arbitrary weights.
%%%%%%%%%%%%%%%%%%%%%%% THEOREM THEOREM THEOREM %%%%%%%%%%%%%%%%%%%%%
\begin{theorem}\label{thm:FS}
Let $1<p<\infty$, and let $S_{\alpha, L}$ be defined in \eqref{def:SaL} with $\alpha \ge 1$ and $L$ satisfying {\rm (A1)} and {\rm (A2)}. Then for every weight $w$,
\begin{align}
\label{eq:SMw-1} \|S_{\alpha,L}(f)\|_{L^p(w)} &\lesssim \alpha^n \|f\|_{L^p(Mw)}, \quad 1<p \le 2,
\\%%%%%%%%%%%%
\label{eq:SMw-2} \|S_{\alpha,L}(f)\|_{L^p(w)} &\lesssim \alpha^n \|f (Mw/w)^{\frac12}\|_{L^p(w)}, \quad 2<p<\infty,
\end{align}
where the implicit constants are independent of $w$ and $f$.
\end{theorem}
%%%%%%%%%%%%%%%%%%%%%%% THEOREM THEOREM THEOREM %%%%%%%%%%%%%%%%%%%%%

We turn to some weak type estimates for $S_{\alpha, L}$.
%%%%%%%%%%%%%%%%%%%%%%% THEOREM THEOREM THEOREM %%%%%%%%%%%%%%%%%%%%%
\begin{theorem}\label{thm:local}
Let $S_{\alpha, L}$ be defined in \eqref{def:SaL} with $\alpha \ge 1$ and $L$ satisfying {\rm (A1)} and {\rm (A2)}. Let $B \subset X$ be a ball and every function $f \in L^{\infty}_c(\Rn)$ with $\supp (f) \subset B$. Then there exist constants $c_1>0$ and $c_2>0$ such that
\begin{align}\label{eq:local}
| \big\{x \in B:  S_{\alpha,L}(f)(x) > t M(f)(x)  \big\}| \leq c_1 e^{- c_2 t^2}  |B|, \quad \forall t>0.
\end{align}
\end{theorem}
%%%%%%%%%%%%%%%%%%%%%%% THEOREM THEOREM THEOREM %%%%%%%%%%%%%%%%%%%%%

%%%%%%%%%%%%%%%%%%%%%%% THEOREM THEOREM THEOREM %%%%%%%%%%%%%%%%%%%%%
\begin{theorem}\label{thm:mixed}
Let $S_{\alpha, L}$ be defined in \eqref{def:SaL} with $\alpha \ge 1$ and $L$ satisfying {\rm (A1)} and {\rm (A2)}. If $u$ and $v$ satisfy
\begin{align*}
(1) \quad u \in A_{1} \text{ and } uv \in A_{\infty},\ \
\text{ or }\quad  (2)\quad u  \in A_1 \text{ and } v \in A_{\infty},
\end{align*}
then we have
\begin{align}
\bigg\|\frac{S_{\alpha,L}(f)}{v}\bigg\|_{L^{1,\infty}(uv)} \lesssim   \| f \|_{L^1(u)},
\end{align}
In particular, $S_{\alpha,L}$ is bounded from $L^1(u)$ to $L^{1, \infty}(u)$ for every $u \in A_{1}$.
\end{theorem}
%%%%%%%%%%%%%%%%%%%%%%% THEOREM THEOREM THEOREM %%%%%%%%%%%%%%%%%%%%%

Given $1 \le p<\infty$, $A_p^{\mathcal{R}}$ denotes the class of weights $w$ such that
\begin{align*}
[w]_{A_p^{\mathcal{R}}} := \sup_{E \subset Q} \frac{|E|}{|Q|} \bigg(\frac{w(Q)}{w(E)}\bigg)^{\frac1p}<\infty,
\end{align*}
where the supremum is taken over all cubes $Q$ and all measurable sets $E \subset Q$. This $A_p^{\mathcal{R}}$ class was introduced in \cite{KT} to characterize the restricted weak-type $(p, p)$ of the Hardy-Littlewood maximal operator $M$ as follows:
\begin{align}\label{eq:ME}
\|M\mathbf{1}_E\|_{L^{p,\infty}(w)} \lesssim [w]_{A_p^{\mathcal{R}}} w(E)^{\frac1p}.
\end{align} 
We should mention that $A_p \subsetneq  A_p^{\mathcal{R}}$ for any $1<p<\infty$.

%%%%%%%%%%%%%%%%%%%%%%%% THEOREM THEOREM THEOREM %%%%%%%%%%%%%%%%%%%%
\begin{theorem}\label{thm:RW}
Let $S_{\alpha, L}$ be defined in \eqref{def:SaL} with $\alpha \ge 1$ and $L$ satisfying {\rm (A1)} and {\rm (A2)}. Then for every $2<p<\infty$, for every $w \in A_p^{\mathcal{R}}$, and for every measurable set $E \subset \Rn$, 
\begin{align}\label{eq:SLE}
\|S_{\alpha, L}(\mathbf{1}_E)\|_{L^{p,\infty}(w)} \lesssim [w]_{A_p^{\mathcal{R}}}^{1+\frac{p}{2}} w(E)^{\frac1p},
\end{align}
where the implicit constants are independent of $w$ and $E$.
\end{theorem}
%%%%%%%%%%%%%%%%%%%%%%% THEOREM THEOREM THEOREM %%%%%%%%%%%%%%%%%%%%%

Finally, we obtain the endpoint estimate for commutators of $S_{\alpha, L}$ as follows. Given an operator $T$ and measurable functions $b$, we define, whenever it makes sense, the commutator by
\begin{align*}
C_{b}(T)(f)(x) := T((b(x)-b(\cdot))f(\cdot))(x).
\end{align*}

%%%%%%%%%%%%%%%%%%%%%%%% THEOREM THEOREM THEOREM %%%%%%%%%%%%%%%%%%%%
\begin{theorem}\label{thm:SbA1}
Let $S_{\alpha, L}$ be defined in \eqref{def:SaL} with $\alpha \ge 1$ and $L$ satisfying {\rm (A1)} and {\rm (A2)}. Then for every $w \in A_1$,
\begin{equation}
w(\{x\in \Rn: C_b(S_{\alpha,L})f(x)>t\}) \lesssim \int_{\Rn} \Phi \Big(\frac{|f(x)|}{t}\Big) w(x) dx, \quad\forall t>0,
\end{equation}
where $\Phi(t)=t(1+\log^{+}t)$.
\end{theorem}
%%%%%%%%%%%%%%%%%%%%%%%% THEOREM THEOREM THEOREM %%%%%%%%%%%%%%%%%%%%

%%%%%%%%%%%%%%%%%%%%% SECTION SECTION SECTION %%%%%%%%%%%%%%%%%%%%%%%%%
%%%%%%%%%%%%%%%%%%%%% SECTION SECTION SECTION %%%%%%%%%%%%%%%%%%%%%%%%%
\section{Applications}\label{sec:app}
The goal of this section is to give some applications of Theorems \ref{thm:Suv}--\ref{thm:SbA1}. To this end, we introduce some new operators. Associated with $L$ introduced in Section \ref{Introduction}, we can also define the square functions $g_{L}$ and $g^*_{\lambda,L}$ ($\lambda>0$) as follows:
\begin{align*}
g_{L}(f)(x) &:=\bigg(\int_0^\infty |t^m Le^{-t^m L}f(x)|^2 \frac{dt}{t}\bigg)^{\frac12},
\\%%%%%%%%%%%%%%%
g^*_{\lambda,L}(f)(x) &:=\bigg(\int_0^\infty\int_{\Rn}\bigg( \frac{t}{t+|x-y|}\bigg)^{n\lambda} |t^mLe^{-t^mL}f(y)|^2\frac{dydt}{t^{n+1}}\bigg)^{\frac12}.
\end{align*}
If $L$  satisfies (A1) and (A2), we have the following estimates (cf. \cite[p. 891]{BD}):
\begin{align}
\label{eq:gL-1} g_{L}(f)(x) &\lesssim  g^*_{\lambda, L}(f)(x), \quad x \in \Rn,
\\
\label{eq:gL-2} g^*_{\lambda,L}(f)(x) &\lesssim \sum_{k=0}^\infty 2^{-k\lambda n/ 2}S_{2^k, L}f(x),\quad x \in \Rn,
\end{align}
whenever $\lambda>2$. By \eqref{eq:gL-1}, \eqref{eq:gL-2} and Theorems \ref{thm:Suv}--\ref{thm:SbA1}, we conclude the following:

%%%%%%%%%%%%%%%%%%%%%%%% THEOREM THEOREM THEOREM %%%%%%%%%%%%%%%%%%%%
\begin{theorem}\label{thm:app-1}
Let $L$ satisfy {\rm (A1)} and {\rm (A2)}. Then Theorems \ref{thm:Suv}--\ref{thm:SbA1} are also true for $g_L$ and $g^*_{\lambda,L}$, whenever $\lambda>2$.
\end{theorem}
%%%%%%%%%%%%%%%%%%%%%%%% THEOREM THEOREM THEOREM %%%%%%%%%%%%%%%%%%%%

Next, we introduce a class of square functions associated to $L$ and $D$, where $D$ is an operator which plays the role of the directional derivative or gradient operator. Assume that $m$ is a positive even integer. Let $D$ be a densely defined linear operator on $L^2(\mathbb{R}^n)$ which possess the following properties:
\begin{enumerate}
\item[(D1)] $D^{m/ 2}L^{-1/ 2}$ is bounded on $L^2(\mathbb{R}^n)$;
\vspace{0.2cm}
\item[(D2)]there exist $c_1,  c_2 > 0$ such that
$$|D^{m/ 2}p_t(x,y)|\leq \frac{c_1}{\sqrt{t}|B(x,t^{1/ m})|}\exp\bigg(-\frac{|x-y|^{m/(m-1)}}{c_2 \, t^{1/(m-1)}}\bigg).$$
\end{enumerate}
Given $\alpha \ge 1$ and $\lambda>2$, we define the following square functions associated to $L$ and $D$:
\begin{align*}
g_{D, L}(f)(x) &:=\bigg(\int_0^\infty |t^{\frac{m}{2}} D^{\frac{m}{2}}e^{-t^m L}f(x)|^2 \frac{dt}{t}\bigg)^{\frac12},
\\%%%%%%%%%%%%%%%
S_{\alpha,D, L}(f)(x) &:=\bigg(\iint_{\Gamma_{\alpha}(x)} |t^{\frac{m}{2}} D^{\frac{m}{2}}e^{-t^mL}f(y)|^2\frac{dydt}{t^{n+1}}\bigg)^{\frac12},
\\%%%%%%%%%%%%%%%
g^*_{\lambda,D,L}(f)(x) &:=\bigg(\int_0^\infty\int_{\Rn}\bigg( \frac{t}{t+|x-y|}\bigg)^{n\lambda} |t^{\frac{m}{2}} D^{\frac{m}{2}}e^{-t^mL}f(y)|^2\frac{dydt}{t^{n+1}}\bigg)^{\frac12}.
\end{align*}
It was proved in \cite[p.~895]{BD} that
\begin{align}\label{eq:gDL-1}
g_{D, L}(f)(x) \lesssim g^*_{\lambda, L}(f)(x), \quad x \in \Rn \text{ and } \lambda>2.
\end{align}
On the other hand, we note that $S_{\alpha,D, L}$ has the same properties as $S_{\alpha,L}$ and
\begin{align}\label{eq:gDL-2}
g^*_{\lambda,D,L}(f)(x) &\lesssim \sum_{k=0}^\infty 2^{-k\lambda n/ 2}S_{2^k, D, L}(f)(x),\quad x \in \Rn \text{ and }  \lambda>2.
\end{align}
Then, \eqref{eq:gDL-1}, \eqref{eq:gDL-2} and Theorems \ref{thm:Suv}--\ref{thm:SbA1} give the following.

%%%%%%%%%%%%%%%%%%%%%%%% THEOREM THEOREM THEOREM %%%%%%%%%%%%%%%%%%%%
\begin{theorem}\label{thm:app-2}
Let $L$ satisfy {\rm (A1)} and {\rm (A2)} and $D$ satisfy {\rm (D1)} and {\rm (D2)}. Let $\alpha\geq 1$ and $\lambda>2$. Then Theorems \ref{thm:Suv}--\ref{thm:SbA1} also hold for $g_{D,L}$, $S_{\alpha,D, L}$ and $g^*_{\lambda,D,L}$.
\end{theorem}
%%%%%%%%%%%%%%%%%%%%%%%% THEOREM THEOREM THEOREM %%%%%%%%%%%%%%%%%%%%

Finally, we defined a class of more general square functions. Assume that $L$ is a nonnegative self-adjoint operator in $L^2(\mathbb{R}^n)$ and satisfies (A2). Denote by $E_L (\lambda)$ the spectral decomposition of $L$. Then by spectral theory, for any bounded Borel function $F : [0,\infty)\rightarrow C$ we can define
\[
F(L)=\int_0^\infty F(\lambda)dE_L(\lambda)
\]
as a bounded operator on $L^2(\mathbb{R}^n)$.

Let $\psi$ be an even real-valued function in the Schwartz space $\mathcal{S}(\R)$ such that $\int_0^\infty \psi^2(s)\frac{ds}{s}<\infty$. Given $\alpha \ge 1$ and $\lambda>2$, we now consider the following square functions:
\begin{align*}
g_{\psi, L}(f)(x) &:=\bigg(\int_0^\infty |\psi(t^\frac{m}{2}\sqrt{L})f(x)|^2 \frac{dt}{t}\bigg)^{\frac12},
\\%%%%%%%%%%%%%%%
S_{\alpha,\psi, L}(f)(x) &:=\bigg(\iint_{\Gamma_{\alpha}(x)} |\psi(t^\frac{m}{2}\sqrt{L})f(y)|^2\frac{dydt}{t^{n+1}}\bigg)^{\frac12},
\\%%%%%%%%%%%%%%%
g^*_{\lambda,\psi,L}(f)(x) &:=\bigg(\int_0^\infty\int_{\Rn} \bigg( \frac{t}{t+|x-y|}\bigg)^{n\lambda} |\psi(t^\frac{m}{2}\sqrt{L})f(y)|^2\frac{dydt}{t^{n+1}}\bigg)^{\frac12}.
\end{align*}
Observe that for any $N>0$,
\begin{align}\label{eq:psiL}
|\psi(t^{m/2} \sqrt{t})(x, y)| \le C_N \frac{1}{t^n} \bigg(1+\frac{|x-y|}{t}\bigg)^{-N},\quad t>0, \, x, y \in \Rn.
\end{align}
Using \eqref{eq:psiL} and the argument for $S_{\alpha,L}$, we obtain that the estimates in Section \ref{Introduction} is true for $S_{\alpha,D, L}$. Additionally, for any $\lambda>2$,
\begin{align}
\label{eq:gpsiL-1} g_{\psi, L}f(x) &\lesssim g^*_{\lambda, \varphi, L}(f)(x) + g^*_{\lambda,\psi, L}f(x), \quad x \in \Rn,
\\
\label{eq:gpsiL-2} g^*_{\lambda,\psi,L}(f)(x) &\lesssim \sum_{k=0}^\infty 2^{-k\lambda n/ 2}S_{2^k,\psi, L}(f)(x),\quad x \in \Rn,
\end{align}
where $\varphi \in \S(\R)$ is a fixed function supported in $[2^{-m/2}, 2^{m/2}]$. The proof of \eqref{eq:gpsiL-1} is given in \cite{BD}, while the proof of \eqref{eq:gpsiL-2} is as before. Together with Theorems \ref{thm:Suv}--\ref{thm:SbA1}, these estimates imply the conclusions as follows.

%%%%%%%%%%%%%%%%%%%%%%%% THEOREM THEOREM THEOREM %%%%%%%%%%%%%%%%%%%%
\begin{theorem}\label{thm:app-3}
Let $L$ be a nonnegative self-adjoint operator in $L^2(\mathbb{R}^n)$ and satisfy {\rm (A2)}. Let $\alpha\geq 1$ and $\lambda>2$. Then Theorems \ref{thm:Suv}--\ref{thm:SbA1} are true for $g_{\psi, L}$, $S_{\alpha,\psi, L}$ and $g^*_{\lambda,\psi,L}$.
\end{theorem}
%%%%%%%%%%%%%%%%%%%%%%%% THEOREM THEOREM THEOREM %%%%%%%%%%%%%%%%%%%%

%%%%%%%%%%%%%%%%%%%%% SECTION SECTION SECTION %%%%%%%%%%%%%%%%%%%%%%%%%
%%%%%%%%%%%%%%%%%%%%% SECTION SECTION SECTION %%%%%%%%%%%%%%%%%%%%%%%%%
\section{Preliminaries}\label{sec:pre}

%%%%%%%%%%%%%%%%%%%%%% SUBSECTION SUBSECTION SUBSECTION %%%%%%%%%%%%%%%%%
\subsection{Muckenhoupt class}
By a weight $w$, we mean that $w$ is a nonnegative locally integrable function on $\Rn$. The weight $w$ is said to belong to the Muckenhoupt class $A_p$, $1 < p<\infty$, if
\[
[w]_{A_p} :=\sup_Q \bigg(\fint_{Q}w\, dx \bigg) \bigg(\fint_{Q} w^{-\frac{1}{p-1}}dx\bigg)^{p-1}<\infty,
\]
where the supremum is taken over all cubes in $\Rn$.

\subsection{Dyadic cubes}
Denote by $\ell(Q)$ the sidelength of the cube $Q$. Given a cube $Q_0 \subset \Rn$, let $\D(Q_0)$ denote the set of all dyadic cubes with respect to $Q_0$, that is, the cubes obtained by repeated subdivision of $Q_0$ and each of its descendants into $2^n$ congruent subcubes.

%%%%%%%%%%%%%%%%%%%%% DERFINITION DEFINITION DEFINITION %%%%%%%%%%%%%%%%%%%
\begin{definition}
A collection $\D$ of cubes is said to be a dyadic grid if it satisfies
\begin{enumerate}
\item [(1)] For any $Q \in \D$, $\ell(Q) = 2^k$ for some $k \in \Z$.
\item [(2)] For any $Q,Q' \in \D$, $Q \cap Q' = \{Q,Q',\emptyset\}$.
\item [(3)] The family $\D_k=\{Q \in \D; \ell(Q)=2^k\}$ forms a partition of $\Rn$ for any $k \in \Z$.
\end{enumerate}
\end{definition}
%%%%%%%%%%%%%%%%%%%%% DERFINITION DEFINITION DEFINITION %%%%%%%%%%%%%%%%%%%

%%%%%%%%%%%%%%%%%%%%% DERFINITION DEFINITION DEFINITION %%%%%%%%%%%%%%%%%%%
\begin{definition}
A subset $\S$ of a dyadic grid is said to be $\eta$-sparse, $0<\eta<1$, if for every $Q \in \S$, there exists a measurable set $E_Q \subset Q$ such that $|E_Q| \geq \eta |Q|$, and the sets $\{E_Q\}_{Q \in \S}$ are pairwise disjoint.
\end{definition}
%%%%%%%%%%%%%%%%%%%%% DEFINITION DEFINITION DEFINITION %%%%%%%%%%%%%%%%%%%

By a median value of a measurable function $f$ on a cube $Q$ we mean a possibly non-unique, real number $m_f (Q)$ such that
\[
\max \big\{|\{x \in Q : f(x) > m_f(Q) \}|,
|\{x \in Q : f(x) < m_f(Q) \}| \big\} \leq |Q|/2.
\]
The decreasing rearrangement of a measurable function $f$ on $\Rn$ is
defined by
\[
f^*(t) = \inf \{ \alpha > 0 : |\{x \in \Rn : |f(x)| > \alpha \}| < t \},
\quad 0 < t < \infty.
\]
The local mean oscillation of $f$ is
\[
\omega_{\lambda}(f; Q)
= \inf_{c \in \R} \big( (f-c) \mathbf{1}_{Q} \big)^* (\lambda |Q|),
\quad 0 < \lambda < 1.
\]
Given a cube $Q_0$, the local sharp maximal function is
defined by
\[
M_{\lambda; Q_0}^{\sharp} f (x)
= \sup_{x \in Q \subset Q_0} \omega_{\lambda}(f; Q).
\]

Observe that for any $\delta > 0$ and $0 < \lambda < 1$
\begin{equation}\label{eq:mfQ}
|m_f(Q)| \leq (f \mathbf{1}_Q)^* (|Q|/2) \ \ \text{and} \ \
(f \mathbf{1}_Q)^* (\lambda |Q|) \leq
\left( \frac{1}{\lambda} \fint_{Q} |f|^{\delta} dx \right)^{1/{\delta}}.
\end{equation}
The following theorem was proved by Hyt\"{o}nen \cite[Theorem~2.3]{Hy2} in order to improve Lerner's formula given in \cite{Ler11} by getting rid of the local sharp maximal function.

%%%%%%%%%%%%%%%%%%%%%%%% LEMMA LEMMA LEMMA  %%%%%%%%%%%%%%%%%%%%%%%%
\begin{lemma}\label{lem:mf}
Let $f$ be a measurable function on $\Rn$ and let $Q_0$ be a fixed cube. Then there exists a (possibly empty) sparse family $\S(Q_0) \subset \D(Q_0)$ such that
\begin{equation}\label{eq:mf}
|f (x) - m_f (Q_0)| \leq 2 \sum_{Q \in \S(Q_0)} \omega_{2^{-n-2}}(f; Q) \mathbf{1}_Q (x), \quad a. e. ~ x \in Q_0.
\end{equation}
\end{lemma}

\subsection{Orlicz maximal operators}
A function $\Phi:[0,\infty) \to [0,\infty)$ is called a Young function if it is continuous, convex, strictly increasing, and satisfies
\begin{equation*}
\lim_{t\to 0^{+}}\frac{\Phi(t)}{t}=0 \quad\text{and}\quad \lim_{t\to\infty}\frac{\Phi(t)}{t}=\infty.
\end{equation*}
Given $p \in[1, \infty)$, we say that a Young function $\Phi$ is a  $p$-Young function, if $\Psi(t)=\Phi(t^{1/p})$ is a Young function.

If $A$ and $B$ are Young functions, we write $A(t) \simeq B(t)$ if there are constants $c_1, c_2>0$ such that
$c_1 A(t) \leq B(t) \leq c_2 A(t)$ for all $t \geq t_0>0$. Also, we denote $A(t) \preceq B(t)$ if there exists $c>0$ such that $A(t) \leq B(ct)$ for all $t \geq t_0>0$. Note that for all Young functions $\phi$, $t \preceq \phi(t)$. Further, if $A(t)\leq cB(t)$ for some $c>1$, then by convexity, $A(t) \leq B(ct)$.

A function $\Phi$ is said to be doubling, or $\Phi \in \Delta_2$, if there is a constant $C>0$ such that $\Phi(2t) \leq C \Phi(t)$ for any $t>0$. Given a Young function $\Phi$, its complementary function $\bar{\Phi}:[0,\infty) \to [0,\infty)$ is defined by
\[
\bar{\Phi}(t):=\sup_{s>0}\{st-\Phi(s)\}, \quad t>0,
\]
which clearly implies that
\begin{align}\label{eq:stst}
st \leq \Phi(s) + \bar{\Phi}(t), \quad s, t > 0.
\end{align}
Moreover, one can check that $\bar{\Phi}$ is also a Young function and
\begin{equation}\label{eq:Young-1}
t \leq \Phi^{-1}(t) \bar{\Phi}^{-1}(t) \leq 2t, \qquad t>0.
\end{equation}
In turn, by replacing $t$ by $\Phi(t)$ in first inequality of \eqref{eq:Young-1}, we obtain
\begin{equation}\label{eq:Young-2}
\bar{\Phi} \Big(\frac{\Phi(t)}{t}\Big) \leq \Phi(t), \qquad t>0.
\end{equation}

Given a Young function $\Phi$, we define the Orlicz space $L^{\Phi}(\Omega, u)$ to be the function space with Luxemburg norm
\begin{align}\label{eq:Orlicz}
\|f\|_{L^{\Phi}(\Omega, u)} := \inf\bigg\{\lambda>0:
\int_{\Omega} \Phi \Big(\frac{|f(x)|}{\lambda}\Big) du(x) \leq 1 \bigg\}.
\end{align}
Now we define the Orlicz maximal operator
\begin{align*}
M_{\Phi}f(x) := \sup_{Q \ni x} \|f\|_{\Phi, Q} := \sup_{Q \ni x} \|f\|_{L^{\Phi}(Q, \frac{dx}{|Q|})},
\end{align*}
where the supremum is taken over all cubes $Q$ in $\Rn$. When $\Phi(t)=t^p$, $1\leq p<\infty$,
\begin{align*}
\|f\|_{\Phi, Q} = \bigg(\fint_{Q} |f(x)|^p dx \bigg)^{\frac1p}=:\|f\|_{p, Q}.
\end{align*}
In this case, if $p=1$, $M_{\Phi}$ agrees with the classical Hardy-Littlewood maximal operator $M$; if $p>1$, $M_{\Phi}f=M_pf:=M(|f|^p)^{1/p}$. If $\Phi(t) \preceq \Psi(t)$, then $M_{\Phi}f(x) \leq c M_{\Psi}f(x)$ for all $x \in \Rn$.

The H\"{o}lder inequality can be generalized to the scale of Orlicz spaces \cite[Lemma~5.2]{CMP11}.
%%%%%%%%%%%%%%%%%%%%%%%% LEMMA LEMMA LEMMA %%%%%%%%%%%%%%%%%%%%%%%%
\begin{lemma}
Given a Young function $A$, then for all cubes $Q$,
\begin{equation}\label{eq:Holder-AA}
\fint_{Q} |fg| dx \leq 2 \|f\|_{A, Q} \|g\|_{\bar{A}, Q}.
\end{equation}
More generally, if $A$, $B$ and $C$ are Young functions such that $A^{-1}(t) B^{-1}(t) \leq c_1 C^{-1}(t), $ for all $t \geq t_0>0$,
then
\begin{align}\label{eq:Holder-ABC}
\|fg\|_{C, Q} \leq c_2 \|f\|_{A, Q} \|g\|_{B, Q}.
\end{align}
\end{lemma}
%%%%%%%%%%%%%%%%%%%%%%%% LEMMA LEMMA LEMMA %%%%%%%%%%%%%%%%%%%%%%%%

 The following result is an extension of the well-known Coifman-Rochberg theorem. The proof can be found in \cite[Lemma~4.2]{HP}.
%%%%%%%%%%%%%%%%%%%%%%%% LEMMA LEMMA LEMMA %%%%%%%%%%%%%%%%%%%%%%%%
\begin{lemma}
Let $\Phi$ be a Young function and $w$ be a nonnegative function such that $M_{\Phi}w(x)<\infty$ a.e.. Then
 \begin{align}
 \label{eq:CR-Phi} [(M_{\Phi}w)^{\delta}]_{A_1} &\le c_{n,\delta}, \quad\forall \delta \in (0, 1),
 \\%%%%%%%%%%%%%%
\label{eq:MPhiRH} [(M_{\Phi} w)^{-\lambda}]_{RH_{\infty}} &\le c_{n,\lambda},\quad\forall \lambda>0.
 \end{align}
 \end{lemma}
%%%%%%%%%%%%%%%%%%%%%%%% LEMMA LEMMA LEMMA %%%%%%%%%%%%%%%%%%%%%%%%

Given $p \in (1, \infty)$, a Young function $\Phi$ is said to satisfy the $B_p$ condition (or, $\Phi \in B_p$) if for some $c>0$,
\begin{align}\label{def:Bp}
\int_{c}^{\infty} \frac{\Phi(t)}{t^p} \frac{dt}{t} < \infty.
\end{align}
Observe that if \eqref{def:Bp} is finite for some $c>0$, then it is finite for every $c>0$. Let $[\Phi]_{B_p}$ denote the value if $c=1$ in \eqref{def:Bp}. It was shown in \cite[Proposition~5.10]{CMP11} that if $\Phi$ and $\bar{\Phi}$ are doubling Young functions, then $\Phi \in B_p$ if and only if
\begin{align*}
\int_{c}^{\infty} \bigg(\frac{t^{p'}}{\bar{\Phi}(t)}\bigg)^{p-1} \frac{dt}{t} < \infty.
\end{align*}

Let us present two types of $B_p$ bumps. An important special case is the ``log-bumps" of the form
\begin{align}\label{eq:log}
A(t) =t^p \log(e+t)^{p-1+\delta}, \quad  B(t) =t^{p'} \log(e+t)^{p'-1+\delta},\quad \delta>0.
\end{align}
Another interesting example is the ``loglog-bumps" as follows:
\begin{align}
\label{eq:loglog-1} &A(t)=t^p \log(e+t)^{p-1} \log\log(e^e+t)^{p-1+\delta}, \quad \delta>0\\
\label{eq:loglog-2} &B(t)=t^{p'} \log(e+t)^{p'-1} \log\log(e^e+t)^{p'-1+\delta}, \quad \delta>0.
\end{align}
Then one can verify that in both cases above, $\bar{A} \in B_{p'}$ and $\bar{B} \in B_p$ for any $1<p<\infty$.

The $B_p$ condition can be also characterized by the boundedness of the Orlicz maximal operator $M_{\Phi}$. Indeed, the following result was given in \cite[Theorem~5.13]{CMP11} and \cite[eq. (25)]{HP}.
%%%%%%%%%%%%%%%%%%%%%%%%%% LEMMA LEMMA LEMMA %%%%%%%%%%%%%%%%%%%%%%
\begin{lemma}\label{lem:MBp}
Let $1<p<\infty$. Then $M_{\Phi}$ is bounded on $L^p(\Rn)$ if and only if $\Phi \in B_p$. Moreover, $\|M_{\Phi}\|_{L^p(\Rn) \to L^p(\Rn)} \le C_{n,p} [\Phi]_{B_p}^{\frac1p}$. In particular, if the Young function $A$ is the same as the first one in \eqref{eq:log} or \eqref{eq:loglog-1}, then
\begin{equation}\label{eq:MAnorm}
\|M_{\bar{A}}\|_{L^{p'}(\Rn) \to L^{p'}(\Rn)} \le c_n p^2 \delta^{-\frac{1}{p'}},\quad\forall \delta \in (0, 1].
\end{equation}
\end{lemma}
%%%%%%%%%%%%%%%%%%%%%%%%%% LEMMA LEMMA LEMMA %%%%%%%%%%%%%%%%%%%%%%

%%%%%%%%%%%%%%%%%%%%% DEFINITION DEFINITION DEFINITION %%%%%%%%%%%%%%%%%%%%
\begin{definition}\label{def:sepbum}
Given $p \in (1, \infty)$, let $A$ and $B$ be Young functions such that $\bar{A} \in B_{p'}$ and $\bar{B} \in B_p$. We say that the pair of weights $(u, v)$ satisfies the {\tt double bump condition} with respect to $A$ and $B$ if
\begin{align}\label{eq:uvABp}
[u, v]_{A,B,p}:=\sup_{Q} \|u^{\frac1p}\|_{A,Q} \|v^{-\frac1p}\|_{B,Q} < \infty.
\end{align}
where the supremum is taken over all cubes $Q$ in $\Rn$. Also, $(u, v)$ is said to satisfy the {\tt separated bump condition} if
\begin{align}
\label{eq:uvAp} [u, v]_{A,p'} &:= \sup_{Q} \|u^{\frac1p}\|_{A,Q} \|v^{-\frac1p}\|_{p',Q} < \infty,
\\%%%%%%%%%%
\label{eq:uvpB} [u, v]_{p,B} &:= \sup_{Q} \|u^{\frac1p}\|_{p,Q} \|v^{-\frac1p}\|_{B,Q} < \infty.
\end{align}
\end{definition}
%%%%%%%%%%%%%%%%%%%%% DEFINITION DEFINITION DEFINITION %%%%%%%%%%%%%%%%%%%%

Note that if $A(t)=t^p$ in \eqref{eq:uvAp} or $B(t)=t^p$ in \eqref{eq:uvpB}, each of them actually is two-weight $A_p$ condition and we denote them by $[u, v]_{A_p}:=[u, v]_{p,p'}$.  Also, the separated bump condition is weaker than the double bump condition. Indeed, \eqref{eq:uvABp} implies \eqref{eq:uvAp} and \eqref{eq:uvpB}, but the reverse direction is incorrect.
The first fact holds since $\bar{A} \in B_{p'}$ and $\bar{B} \in B_p$ respectively indicate $A$ is a $p$-Young function and $B$ is a $p'$-Young function. The second fact was shown in \cite[Section~7]{ACM} by constructing log-bumps.

%%%%%%%%%%%%%%%%%%%%%%%% LEMMA LEMMA LEMMA %%%%%%%%%%%%%%%%%%%%%%%%
\begin{lemma}\label{lem:M-uv}
Let $1<p<\infty$, let $A$, $B$ and $\Phi$ be Young functions such that $A \in B_p$ and $A^{-1}(t)B^{-1}(t) \lesssim \Phi^{-1}(t)$ for any $t>t_0>0$. If a pair of weights $(u, v)$ satisfies $[u, v]_{p, B}<\infty$, then
\begin{align}\label{eq:MPhi-uv}
\|M_{\Phi}f\|_{L^p(u)} \leq C [u, v]_{p, B} [A]_{B_p}^{\frac1p} \|f\|_{L^p(v)}.
\end{align}
Moreover, \eqref{eq:MPhi-uv} holds for $\Phi(t)=t$ and $B=\bar{A}$ satisfying the same hypotheses.  In this case, $\bar{A} \in B_p$ is necessary.
\end{lemma}
%%%%%%%%%%%%%%%%%%%%%%%% LEMMA LEMMA LEMMA %%%%%%%%%%%%%%%%%%%%%%%%

The two-weight inequality above was established in \cite[Theorem~5.14]{CMP11} and \cite[Theorem~3.1]{CP99}.  The weak type inequality for $M_{\Phi}$ was also obtained in \cite[Proposition~5.16]{CMP11} as follows.

%%%%%%%%%%%%%%%%%%%%%%%% LEMMA LEMMA LEMMA %%%%%%%%%%%%%%%%%%%%%%%%
\begin{lemma}\label{lem:Muv-weak}
Let $1<p<\infty$, let $B$ and $\Phi$ be Young functions such that $t^{\frac1p} B^{-1}(t) \lesssim \Phi^{-1}(t)$ for any $t>t_0>0$. If a pair of weights $(u, v)$ satisfies $[u, v]_{p, B}<\infty$, then
\begin{align}\label{eq:MPuv}
\|M_{\Phi}f\|_{L^{p,\infty}(u)} \leq C \|f\|_{L^p(v)}.
\end{align}
Moreover, \eqref{eq:MPuv} holds for $M$ if and only if $[u, v]_{A_p}<\infty$.
\end{lemma}

%%%%%%%%%%%%%%%%%%%%%%%%% SECTION SECTION SECTION %%%%%%%%%%%%%%%%%%%%%
%%%%%%%%%%%%%%%%%%%%%%%%% SECTION SECTION SECTION %%%%%%%%%%%%%%%%%%%%%
\section{Proof of main results}

%%%%%%%%%%%%%%%%%%%%% SUBSECTION SUBSECTION SUBSECTION %%%%%%%%%%%%%%%%%%
\subsection{Sparse domination}
Let $\Phi$ be a radial Schwartz function such that $\mathbf{1}_{B(0, 1)} \le \Phi \le \mathbf{1}_{B(0, 2)}$. We define
\begin{align*}
\widetilde{S}_{\alpha,L}(f)(x):=\bigg(\int_{0}^{\infty} \int_{\Rn}
\Phi\Big(\frac{|x-y|}{\alpha t}\Big) |Q_{t,L}f(y)|^2\frac{dydt}{t^{n+1}}\bigg)^{1/2},
\end{align*}
where $Q_{t,L}f:=t^m L e^{-t^m L}f$. It is easy to verify that
\begin{align}\label{eq:SSS}
S_{\alpha,L}(f)(x) \le \widetilde{S}_{\alpha,L}(f)(x) \le S_{2\alpha,L}(f)(x),\quad x \in \Rn.
\end{align}
Additionally, it was proved in \cite{ADM} that $S_{1,L}$ is bounded from $L^1(\mathbb{R}^n)$ to $L^{1,\infty}(\mathbb{R}^n)$. Then, this and \eqref{eq:SSS} give that
\begin{align}\label{eq:S11}
\|\widetilde{S}_{\alpha,L}(f)\|_{L^{1,\infty}(\mathbb{R}^n)} \lesssim \alpha^n \|f\|_{L^1(\mathbb{R}^n)}.
\end{align}

Using these facts, we can establish the sparse domination for $S_{\alpha,L}$ as follows.
%%%%%%%%%%%%%%%%%%%%%%%%%% LEMMA LEMMA LEMMA %%%%%%%%%%%%%%%%%%%%%%
\begin{lemma}\label{lem:S-sparse}
For any $\alpha \geq 1$, we have
\begin{align}\label{eq:S-sparse}
S_{\alpha,L}(f)(x) &\lesssim \alpha^{n} \sum_{j=1}^{3^n} \A_{\S_j}^2 (f)(x) ,\quad \text{a.e. } x \in \Rn,
\end{align}
where
\[
\A_{S}^2(f)(x):= \bigg(\sum_{Q \in \S} \langle |f| \rangle_Q^2 \mathbf{1}_Q(x)\bigg)^{\frac12}.
\]
\end{lemma}
%%%%%%%%%%%%%%%%%%%%%%%%%% LEMMA LEMMA LEMMA %%%%%%%%%%%%%%%%%%%%%%

%%%%%%%%%%%%%%%%%%%%%%%%%%% PROOF PROOF PROOF %%%%%%%%%%%%%%%%%%%%%%
\begin{proof}
Fix $Q_0\in \D$. By \eqref{eq:mfQ}, Kolmogorov's inequality and \eqref{eq:S11}, we have
\begin{align}\label{eq:mfSL}
|m_{\widetilde{S}_{\alpha,L}(f)^2}(Q_0)|
&\lesssim \bigg(\fint_{Q_0} |\widetilde{S}_{\alpha,L}(f \mathbf{1}_{Q_0})|^{\frac12} \, dx\bigg)^4
\nonumber\\%%%%%%%%%%%%%%
&\lesssim  \|\widetilde{S}_{\alpha,L}(f \mathbf{1}_{Q_0})\|^2_{L^{1,\infty}(Q_0, \frac{dx}{|Q_0|})}
\lesssim \alpha^{2n}  \bigg(\fint_{Q_0} |f| \, dx\bigg)^2.
\end{align}
From \cite[Proposition~3.2]{BD}, we obtain that for any dyadic cube $Q\subset \Rn$, $\alpha\geq 1$ and $\lambda \in (0, 1)$,
\begin{equation}\label{eq:osc-SL}
\omega_{\lambda}(\widetilde{S}_{\alpha, L}(f)^2;Q)
\lesssim \alpha^{2n} \sum_{j=0}^\infty 2^{-j\delta} \bigg(\fint_{2^j Q} |f|\, dx\bigg)^2,
\end{equation}
where $\delta \in (0, 1)$ is some constant. Invoking Lemma \ref{lem:mf}, \eqref{eq:mfSL} and \eqref{eq:osc-SL}, one can pick a sparse family $\S(Q_0) \subset \D(Q_0)$ so that
\begin{align}
\widetilde{S}_{\alpha,L}(f)(x)^2
& \lesssim |m_{\widetilde{S}_{\alpha,L}(f)^2}(Q_0)| +
\sum_{Q \in \S(Q_0)}\omega_{\varepsilon}(\widetilde{S}_{\alpha, L}(f)^2;Q) \mathbf{1}_{Q}(x)
\nonumber\\%%%%%%%%%%%%%
&\lesssim \alpha^{2n} \sum_{Q\in \S(Q_0)}\sum_{j=0}^\infty 2^{-j\delta}\langle |f|\rangle_{2^jQ}^2 \mathbf{1}_{Q}(x)
\label{eq:SLQj} \\%%%%%%%%%%%%%
&=: \alpha^{2n} \sum_{j=0}^\infty 2^{-j\delta} \mathcal{T}^2_{\S(Q_0), j}(f)(x)^2, \quad\text{ a.e. } x\in Q_0. \label{eq:SLTS}
\end{align}
where
\begin{align*}
\mathcal{T}^2_{\S,j}(f)(x)
&:=\bigg(\sum_{Q\in \S} \langle |f|\rangle_{2^jQ} \mathbf{1}_{Q}(x)\bigg)^{\frac12}.
\end{align*}
Denote
\begin{align*}
\mathcal{T}_{\S,j}(f, g)(x)
&:=\sum_{Q\in \S} \langle |f|\rangle_{2^jQ} \langle |g|\rangle_{2^jQ} \mathbf{1}_{Q}(x),
\\%%%%%%%%%%%%%%%
\mathcal{A}_{\S}(f, g)(x)
&:=\sum_{Q\in \S} \langle |f|\rangle_{Q} \langle |g|\rangle_{2Q} \mathbf{1}_{Q}(x).
\end{align*}
Then, $\mathcal{T}^2_{\S,j}(f)(x)=\mathcal{T}_{\S,j}(f, f)(x)^{\frac12}$. On the other hand, the arguments in \cite[Sections~11-13]{LN} shows that there exist $3^n$ dyadic grids $S_j\in \D_j, j=1,\ldots,3^n$, such that
\begin{align}\label{eq:TAA}
\sum_{j=0}^\infty 2^{-j\delta} \mathcal{T}^1_{\S(Q_0), j}(f,f)(x)
\lesssim  \sum_{j=1}^{3^n} \A_{\S_j}(f, f)(x)
= \sum_{j=1}^{3^n} \A^2_{\S_j}(f)(x)^2.
\end{align}
Gathering \eqref{eq:SLTS} and \eqref{eq:TAA}, we deduce that
\begin{align*}
\widetilde{S}_{\alpha,L}(f)(x)\lesssim \alpha^{n} \sum_{j=1}^{3^n}\A^2_{\S_j}(f)(x), \quad\text{ a.e. } x\in Q_0.
\end{align*}
Since $\Rn=\bigcup_{Q \in \D} Q$, it leads that
\begin{align*}
S_{\alpha,L}(f)(x) \le \widetilde{S}_{\alpha,L}(f)(x)\lesssim \alpha^{n} \sum_{j=1}^{3^n}\A^2_{\S_j}(f)(x), \quad\text{ a.e. } x\in \Rn.
\end{align*}
This completes our proof.
\end{proof}
%%%%%%%%%%%%%%%%%%%%%%%%%%% END END END PROOF %%%%%%%%%%%%%%%%%%%%%%

%%%%%%%%%%%%%%%%%%%%% SUBSECTION SUBSECTION SUBSECTION %%%%%%%%%%%%%%%%%%
\subsection{Bump conjectures}
In this subsection, we are going to show two-weight inequalities invoking bump conjectures.

%%%%%%%%%%%%%%%%%%%%%%%%%%% PROOF PROOF PROOF %%%%%%%%%%%%%%%%%%%%%%
\begin{proof}[\textbf{Proof of Theorem \ref{thm:Suv}.}]
By Lemma \ref{lem:S-sparse}, the inequality \eqref{eq:SLp} follows from the following
\begin{align}\label{eq:ASLp}
\|\A_{\mathcal{S}}^2(f)\|_{L^p(u)} \lesssim \mathscr{N}_p \|f\|_{L^p(v)},
\end{align}
for every sparse family $\mathcal{S}$, where the implicit constant does not depend on $\mathcal{S}$.

To prove \eqref{eq:ASLp}, we begin with the case $1<p \le 2$. Actually, the H\"{o}lder's inequality \eqref{eq:Holder-AA} gives that
\begin{align}\label{eq:ASp1}
\|\A_{\mathcal{S}}^2(f)\|_{L^p(u)}^p
&=\int_{X} \bigg(\sum_{Q \in \mathcal{S}} \langle f \rangle_Q^2 \mathbf{1}_{Q}(x)\bigg)^{\frac{p}{2}} u(x) dx
\le \sum_{Q \in \mathcal{S}} \langle |f| \rangle_Q^p \int_Q u(x)dx
\nonumber \\%%%%%%%%%%%%%%
&\lesssim \sum_{Q \in \mathcal{S}} \|f v^{\frac1p}\|_{\bar{B}, Q}^p \|v^{-\frac1p}\|_{B, Q}^p
\|u^{\frac1p}\|_{p, Q}^p |Q|
\nonumber \\%%%%%%%%%%%%%%
&\lesssim ||(u, v)||_{A,B,p}^p \sum_{Q \in \mathcal{S}} \left(\inf_{Q} M_{\bar{B}}(f v^{\frac1p})\right)^p |E_Q|
\nonumber \\%%%%%%%%%%%%%%
&\le ||(u, v)||_{A,B,p}^p \int_{X} M_{\bar{B}}(f v^{\frac1p})(x)^p dx
\nonumber \\%%%%%%%%%%%%%%
&\le ||(u, v)||_{A,B,p}^p \|M_{\bar{B}}\|_{L^p}^p \|f\|_{L^p(v)}^p,
\end{align}
where Lemma \ref{lem:MBp} is used in the last step.

Next let us deal with the case $2<p<\infty$. By duality, one has
\begin{align}\label{eq:AS-dual}
\|\A_{\mathcal{S}}^2(f)\|_{L^p(u)}^2 = \|\A_{\mathcal{S}}^2(f)^2\|_{L^{p/2}(u)}
=\sup_{\substack{0 \le h \in L^{(p/2)'}(u) \\ \|h\|_{L^{(p/2)'}(u)=1}}} \int_{\Rn} \A_{\mathcal{S}}^2(f)^2 h u\, dx.
\end{align}
Fix a nonnegative function $h \in L^{(p/2)'}(u)$ with $\|h\|_{L^{(p/2)'}(u)}=1$. Then using H\"{o}lder's inequality \eqref{eq:Holder-AA} three times and Lemma \ref{lem:MBp}, we obtain
\begin{align}\label{eq:ASp2}
&\int_{X} \A_{\mathcal{S}}^2(f)(x)^2 h(x) u(x) dx
\lesssim \sum_{Q \in \mathcal{S}} \langle |f| \rangle_{Q}^2 \langle hu \rangle_{Q} |Q|
\nonumber \\%%%%%%%%%%%%%%
&\lesssim \sum_{Q \in \mathcal{S}} \|f v^{\frac1p}\|_{\bar{B}, Q}^2 \|v^{-\frac1p}\|_{B, Q}^2
\|hu^{1-\frac{2}{p}}\|_{\bar{A}, Q} \|u^{\frac{2}{p}}\|_{A, Q} |Q|
\nonumber \\%%%%%%%%%%%%%%
&\lesssim ||(u, v)||_{A,B,p}^2 \sum_{Q \in \mathcal{S}} \left(\inf_{Q} M_{\bar{B}}(f v^{\frac1p})\right)^2
\left(\inf_{Q} M_{\bar{A}}(hu^{1-\frac{2}{p}})\right) |E_Q|
\nonumber \\%%%%%%%%%%%%%%
&\le ||(u, v)||_{A,B,p}^2 \int_{X} M_{\bar{B}}(f v^{\frac1p})(x)^2 M_{\bar{A}}(hu^{1-\frac{2}{p}})(x) dx
\nonumber \\%%%%%%%%%%%%%%
&\le ||(u, v)||_{A,B,p}^2 \|M_{\bar{B}}(f v^{\frac1p})^2\|_{L^{p/2}}
\|M_{\bar{A}}(hu^{1-\frac{2}{p}})\|_{L^{(p/2)'}}
\nonumber \\%%%%%%%%%%%%%%
&\le ||(u, v)||_{A,B,p}^2 \|M_{\bar{B}}\|_{L^p}^2
\|M_{\bar{A}}\|_{L^{(p/2)'}} \|f\|_{L^p( v)}^2 \|h\|_{L^{(p/2)'}(u)}.
\end{align}
Therefore, \eqref{eq:ASLp} immediately follows from \eqref{eq:ASp1}, \eqref{eq:AS-dual} and \eqref{eq:ASp2}.
\end{proof}
%%%%%%%%%%%%%%%%%%%%%%%%%%% END END END PROOF %%%%%%%%%%%%%%%%%%%%%%

Let us present an endpoint extrapolation theorem from \cite[Corollary~8.4]{CMP11}.
%%%%%%%%%%%%%%%%%%%%%%%%%% LEMMA LEMMA LEMMA %%%%%%%%%%%%%%%%%%%%%%
\begin{lemma}
Let $\mathcal{F}$ be a collection of pairs $(f, g)$ of nonnegative measurable functions. If for every weight $w$,
\begin{align*}
\|f\|_{L^{1,\infty}(w)} \le C \|g\|_{L^1(Mw)}, \quad (f, g) \in \mathcal{F},
\end{align*}
then for all $p \in (1, \infty)$,
\begin{align*}
\|f\|_{L^{p,\infty}(u)} \le C \|g\|_{L^p(v)}, \quad (f, g) \in \mathcal{F},
\end{align*}
whenever $\sup_{B} \|u^{\frac1p}\|_{A, B} \|v^{-\frac1p}\|_{p', B}<\infty$, where $\bar{A} \in B_{p'}$.
\end{lemma}
%%%%%%%%%%%%%%%%%%%%%%%%%% LEMMA LEMMA LEMMA %%%%%%%%%%%%%%%%%%%%%%

%%%%%%%%%%%%%%%%%%%%%%%%%%% PROOF PROOF PROOF %%%%%%%%%%%%%%%%%%%%%%
\begin{proof} [\textbf{Proof of Theorem \ref{thm:Sweak}.}]
In view of, it suffices to prove that for every weight $w$,
\begin{align*}
\|S_{\alpha, L}(f)\|_{L^{1,\infty}(w)} \leq C \|f\|_{L^1(Mw)},
\end{align*}
where the constant $C$ is independent of $w$ and $f$. We should mention that although the norm of weights does not appear in \cite[Corollary~8.4]{CMP11}, one can check its proof to obtain the norm constant in \eqref{eq:S-weak}.  Invoking Lemma \ref{lem:S-sparse}, we are reduced to showing that there exists a constant $C$ such that for every sparse family $\S$ and for every weight $w$,
\begin{align}\label{eq:S-11}
\|\A_{\S}^2(f)\|_{L^{1,\infty}(w)} \leq C \|f\|_{L^1(M_{\D}w)},
\end{align}

Without loss of generality, we may assume that $f$ is bounded and has compact support. Fix $\lambda>0$ and denote $\Omega:=\{x \in \Rn: M_{\D}f(x)>\lambda\}$. By the Calder\'{o}n-Zygmund decomposition, there exists a pairwise disjoint family $\{Q_j\} \subset \D$ such that $\Omega=\bigcup_{j}Q_j$ and
\begin{list}{\textup{(\theenumi)}}{\usecounter{enumi}\leftmargin=1cm \labelwidth=1cm \itemsep=0.2cm
			\topsep=.2cm \renewcommand{\theenumi}{\arabic{enumi}}}
\item\label{CZ-1} $f=g+b$,
\item\label{CZ-2} $g=f\mathbf{1}_{\Omega}+\sum_{j} \langle f \rangle_{Q_j} \mathbf{1}_{Q_j}$,
\item\label{CZ-3} $b=\sum_{j}b_j$ with $b_j=(f-\langle f \rangle_{Q_j}) \mathbf{1}_{Q_j}$,
\item\label{CZ-4} $\langle |f| \rangle_{Q_j}>\lambda$ and $|g(x)| \le 2^n \lambda$, a.e. $x \in \Rn$,
\item\label{CZ-5} $\supp(b_j) \subset Q_j$ and $\fint_{Q_j} b_j\, dx=0$.
\end{list}
Then by \eqref{CZ-1}, we split
\begin{align}\label{eq:IgIb}
&w(\{x \in \Rn: \A_{\S}^2(f)(x)>\lambda\})
\le w(\Omega) + {\rm I_g} + {\rm I_b},
\end{align}
where
\begin{align*}
{\rm I_g} = w(\{x \in \Omega^c: \A_{\S}^2(g)(x)>\lambda/2\}) \quad\text{and}\quad
{\rm I_b} =w(\{x \in \Omega^c: \A_{\S}^2(b)(x)>\lambda/2\})
\end{align*}
For the first term, we by \eqref{CZ-4} have
\begin{align}\label{eq:wo}
w(\Omega) &\le \sum_{j} w(Q_j) \le \frac{1}{\lambda} \sum_{j} \frac{w(Q_j)}{|Q_j|} \int_{Q_j} |f(x)| dx
\nonumber \\%%%%%%%%%%%%%
&\le \frac{1}{\lambda} \sum_{j} \int_{Q_j} |f(x)| M_{\D}w(x) dx
\le \frac{1}{\lambda} \int_{\Rn} |f(x)| M_{\D}w(x) dx.
\end{align}

To estimate ${\rm I_b}$, we claim that $\A_{\S}^2(b_j)(x)=0$ for all $x \in \Omega^c$ and for all $j$. In fact, if there exist $x_0 \in \Omega^c$ and $j_0$ such that $\A_{\S}^2(b_{j_0})(x_0) \neq 0$, then there is a dyadic cube $Q_0 \in \S$ such that $x_0 \in Q_0$ and $\langle b_{j_0}\rangle_{Q_0} \neq 0$. The latter implies $Q_0 \subsetneq Q_{j_0}$ because of the support and the vanishing property of $b_{j_0}$. This in turn gives that $x_0 \in Q_{j_0}$, which contradicts $x_0 \in \Omega^c$. This shows our claim. As a consequence, the set $\{x \in \Omega^c: \A_{\S}^2(b)(x)>\lambda/2\}$ is empty, and hence ${\rm I_b}=0$.

In order to control ${\rm I_g}$, we first present a Fefferman-Stein inequality for $\A_{\S}^2$. Note that $v(x):=M_{\D}w(x) \ge \langle w \rangle_{Q}$ for any dyadic cube $Q \in \S$ containing $x$. Then for any Young function $A$ such that $\bar{A} \in B_2$,
\begin{align}\label{eq:AS-FS}
\|\A_{\S}^2(f)\|_{L^2(w)}^2
&=\sum_{Q \in \S} \langle |f| \rangle_{Q}^2 w(Q)
\le \sum_{Q \in \S} \|f v^{\frac12}\|_{\bar{A}, Q}^2 \|v^{-\frac12}\|_{A, Q}^2 w(Q)
\nonumber \\%%%%%%%%%%%%%
&\le \sum_{Q \in \S} \|f v^{\frac12}\|_{\bar{A}, Q}^2 \langle w \rangle_{Q}^{-1} w(Q)
\lesssim \sum_{Q \in \S} \Big(\inf_{Q} M_{\bar{A}}(f v^{\frac12}) \Big)^2 |E_Q|
\nonumber \\%%%%%%%%%%%%%
&\le \int_{\Rn} M_{\bar{A}}(f v^{\frac12})(x)^2 dx
\lesssim \|f\|_{L^2(v)}^2 = \|f\|_{L^2(M_{\D}w)}^2,
\end{align}
where we used Lemma \ref{lem:MBp} in the last inequality.  Note that for any $x \in Q_j$,
\begin{align}\label{eq:MD}
M_{\D}(w\mathbf{1}_{\Omega^c})(x)
&\le M_{\D}(w\mathbf{1}_{Q_j^c})(x)
=\sup_{x \in Q \in \D} \frac{1}{|Q|} \int_{Q \cap Q_j^c} w(y) dy
\nonumber \\%%%%%%%%%%%%%
&\le \sup_{Q \in \D:Q_j \subsetneq Q} \frac{1}{|Q|} \int_{Q \cap Q_j^c} w(y) dy
\le \inf_{Q_j} M_{\D}w.
\end{align}
Thus, combining \eqref{eq:AS-FS} with \eqref{eq:MD}, we have
\begin{align}\label{eq:Ig}
{\rm I_g} &\le \frac{4}{\lambda^2} \int_{\Omega^c} \A_{\S}^2(g)(x) w(x) dx
\nonumber \\%%%%%%%%%%%%%
&\lesssim \frac{1}{\lambda^2} \int_{\Rn} |g(x)| M_{\D}(w\mathbf{1}_{\Omega^c})(x) dx
\lesssim \frac{1}{\lambda} \int_{\Rn} |g(x)| M_{\D}(w\mathbf{1}_{\Omega^c})(x) dx
\nonumber \\%%%%%%%%%%%%%
&\le \frac{1}{\lambda} \int_{\Omega^c} |f(x)| M_{\D}(w\mathbf{1}_{\Omega^c})(x) dx
+ \frac{1}{\lambda} \sum_{j} \langle f \rangle_{Q_j} \int_{Q_j} M_{\D}(w\mathbf{1}_{\Omega^c})(x) dx
\nonumber \\%%%%%%%%%%%%%
&\le \frac{1}{\lambda} \int_{\Rn} |f(x)| M_{\D}w(x) dx
+ \frac{1}{\lambda} \sum_{j} \int_{Q_j} |f(y)| M_{\D}w(y) dy
\nonumber \\%%%%%%%%%%%%%
&\lesssim \frac{1}{\lambda} \int_{\Rn} |f(x)| M_{\D}w(x) dx.
\end{align}
Consequently, gathering \eqref{eq:IgIb}, \eqref{eq:wo} and \eqref{eq:Ig}, we conclude that \eqref{eq:S-11} holds.
\end{proof}

%%%%%%%%%%%%%%%%%%%%% SUBSECTION SUBSECTION SUBSECTION %%%%%%%%%%%%%%%%%%
\subsection{Fefferman-Stein inequalities}
In order to show Theorem \ref{thm:FS}, we recall an extrapolation theorem for arbitrary weights in \cite[Theorem 1.3]{CP}, or \cite[Theorem~4.11]{CO} in the general Banach function spaces.

%%%%%%%%%%%%%%%%%%%%%%%%%% LEMMA LEMMA LEMMA %%%%%%%%%%%%%%%%%%%%%%
\begin{lemma}\label{lem:extra}
Let $\mathcal{F}$ be a collection of pairs $(f, g)$ of nonnegative measurable functions. If for some $p_0 \in (0, \infty)$ and for every weight $w$,
\begin{align*}
\|f\|_{L^{p_0}(w)} \le C \|g\|_{L^{p_0}(Mw)}, \quad (f, g) \in \mathcal{F},
\end{align*}
then for every $p \in (p_0, \infty)$ and for every weight $w$,
\begin{align*}
\|f\|_{L^p(w)} \le C \|g(Mw/w)^{\frac{1}{p_0}}\|_{L^p(w)}, \quad (f, g) \in \mathcal{F}.
\end{align*}
\end{lemma}
%%%%%%%%%%%%%%%%%%%%%%%%%% LEMMA LEMMA LEMMA %%%%%%%%%%%%%%%%%%%%%%

%%%%%%%%%%%%%%%%%%%%%%%%%% PROOF PROOF PROOF %%%%%%%%%%%%%%%%%%%%%%%
\begin{proof}[\textbf{Proof of Theorem \ref{thm:FS}.}]
 Note that $v(x):=Mw(x) \ge \langle w\rangle_{Q}$ for any dyadic cube $Q \in \mathcal{S}$ containing $x$. Let $A$ be a Young function such that ${\bar{A}}\in B_{p}$.
By Lemma \ref{lem:MBp}, we have
\begin{align}\label{eq:MAf}
\|M_{\bar{A}} (f v^{\frac{1}{p}})\|_{L^{p} } \lesssim \|f\|_{L^{p}(v)}.
\end{align}
Thus, using Lemma \ref{lem:S-sparse}, H\"{o}lder's inequality and \eqref{eq:MAf}, we deduce that
\begin{align*}
\|S_{\alpha,L}(\vec{f})\|_{L^p(w)}^p
&\lesssim \alpha^{pn}\sum_{j=1}^{K_0} \sum_{Q \in \mathcal{S}_j}
 \langle |f| \rangle_{Q}^p \int_Q w(x)dx
\\%%%%%%%%%%%%%
&\le \alpha^{pn}\sum_{j=1}^{K_0} \sum_{Q \in \mathcal{S}_j}
\|f \nu^{\frac{1}{p}}\|_{\bar{A}, Q}^p
\|\nu^{-\frac{1}{p}}\|_{A, Q}^p \int_Q w(x)dx
\\%%%%%%%%%%%%%
&\le \alpha^{pn}\sum_{j=1}^{K_0} \sum_{Q \in \mathcal{S}_j}
 \|f \nu^{\frac{1}{p}}\|_{\bar{A}, Q}^p
|Q|
\\%%%%%%%%%%%%%
&\lesssim \alpha^{pn}\sum_{j=1}^{K_0} \sum_{Q \in \mathcal{S}_j}
\Big(\inf_{Q} M_{\bar{A}}(f \nu^{\frac{1}{p}}) \Big)^p |E_Q|
\\%%%%%%%%%%%%%
&\lesssim\alpha^{pn}\int_{\mathbb{R}^n}
 M_{\bar{A}}(f \nu^{\frac{1}{p}})(x)^p dx
\\%%%%%%%%%%%%%
&\lesssim \alpha^{pn}\|f\|_{L^{p}(v)}^p
= \alpha^{pn} \|f\|_{L^{p}(Mw)}^p.
\end{align*}
This shows \eqref{eq:SMw-1}. Finally, \eqref{eq:SMw-2} is a consequence of \eqref{eq:SMw-1} and Lemma \ref{lem:extra}.
\end{proof}

%%%%%%%%%%%%%%%%%%%%%% SUBSECTION SUBSECTION SUBSECTION %%%%%%%%%%%%%%%%%
\subsection{Local decay estimates}

To show Theorem \ref{thm:local}, we need the following Carleson embedding theorem from \cite[Theorem~4.5]{HP}.
%%%%%%%%%%%%%%%%%%%%%%%% LEMMA LEMMA LEMMA %%%%%%%%%%%%%%%%%%%%%%%%
\begin{lemma}\label{lem:Car-emb}
Suppose that the sequence $\{a_Q\}_{Q \in \mathcal{D}}$ of nonnegative numbers satisfies the Carleson packing condition
\begin{align*}
\sum_{Q \in \mathcal{D}:Q \subset Q_0} a_Q \leq A w(Q_0),\quad\forall Q_0 \in \mathcal{D}.
\end{align*}
Then for all $p \in (1, \infty)$ and $f \in L^p(w)$,
\begin{align*}
\bigg(\sum_{Q \in \mathcal{D}} a_Q \Big(\frac{1}{w(Q)} \int_{Q} f(x) w(x) \ du(x)\Big)^p \bigg)^{\frac1p}
\leq A^{\frac1p} p' \|f\|_{L^p(w)}.
\end{align*}
\end{lemma}
%%%%%%%%%%%%%%%%%%%%%%%% LEMMA LEMMA LEMMA %%%%%%%%%%%%%%%%%%%%%%%%

%%%%%%%%%%%%%%%%%%%%%%%% LEMMA LEMMA LEMMA %%%%%%%%%%%%%%%%%%%%%%%%
\begin{lemma}\label{lem:CF-local}
For every $1<p<\infty$ and $w \in A_p$, we have
\begin{align}\label{eq:CF-local}
\|S_{\alpha,L}f\|_{L^2(B, w)} \leq c_{n,p} \alpha^{n}[w]_{A_p}^{\frac12} \|Mf\|_{L^2(B, w)},
\end{align}
for every ball $B \subset \mathbb{R}^n$ and $f\in L_c^\infty(\mathbb{R}^n)$ with $\supp(f) \subset B$.
\end{lemma}
%%%%%%%%%%%%%%%%%%%%%%%% LEMMA LEMMA LEMMA %%%%%%%%%%%%%%%%%%%%%%%%

%%%%%%%%%%%%%%%%%%%%%%% PROOF PROOF PROOF %%%%%%%%%%%%%%%%%%%%%%%%%%
\begin{proof}
Let $w \in A_p$ with $1<p<\infty$. Fix a ball $B \subset \mathbb{R}^n$. From \eqref{eq:SLQj}, there exists a sparse family $\S(Q) \subset \D(Q)$ so that
\begin{align*}
\widetilde{S}_{\alpha,L}(f)(x)^2
&\lesssim \alpha^{2n} \sum_{Q'\in \S(Q)} \sum_{j=0}^\infty 2^{-j\delta} \langle |f|\rangle_{2^jQ'}^2 \mathbf{1}_{Q'}(x)
\\%%%%%%%%%%%%%
&\lesssim \alpha^{2n} \sum_{Q'\in \S(Q)}  \inf_{Q'} M(f)^2 \mathbf{1}_{Q'}(x), \quad\text{ a.e. } x\in Q.
\end{align*}
This implies that
\begin{align*}
\|\widetilde{S}_{\alpha,L}(f)\|_{L^2(Q, w)}^2
&\lesssim \alpha^{2n}\sum_{Q' \in \S(Q)} \inf_{Q'} M(f)^2 w(Q').
\end{align*}
From this and \eqref{eq:SSS}, we see that to obtain \eqref{eq:CF-local}, it suffices to prove
\begin{align}\label{eq:QSQ}
\sum_{Q' \in \S(Q)} \inf_{Q'} M(f)^2 w(Q') \lesssim [w]_{A_p} \|M(f)\|_{L^2(Q, w)}^2.
\end{align}

Recall that a new version of $A_{\infty}$ was introduced by Hyt\"{o}nen and P\'{e}rez \cite{HP}:
\begin{align*}
[w]'_{A_{\infty}} := \sup_{Q} \frac{1}{w(Q)} \int_{Q} M(w \mathbf{1}_Q)(x) dx.
\end{align*}
By \cite[Proposition~2.2]{HP}, there holds
\begin{align}\label{eq:AiAi}
c_n [w]'_{A_{\infty}} \le [w]_{A_{\infty}} \leq [w]_{A_p}.
\end{align}
Observe that for every $Q'' \in \mathcal{D}$,
\begin{align*}
\sum_{Q' \in \mathcal{S}(Q): Q' \subset Q''} w(Q')
&=\sum_{Q' \in \mathcal{S}(Q): Q' \subset Q''} \langle w \rangle_{Q'} |Q'|
\lesssim \sum_{Q' \in \mathcal{S}(Q): Q' \subset Q''} \inf_{Q'} M(w \mathbf{1}_{Q''}) |E_{Q'}|
\\%%%%%%%%%%%%%%%
&\lesssim \int_{Q''} M(w \mathbf{1}_{Q''})(x) dx
\leq [w]'_{A_{\infty}} w(Q'') \lesssim [w]_{A_p} w(Q''),
\end{align*}
where we used the disjointness of $\{E_{Q'}\}_{Q' \in \mathcal{S}(Q)}$ and \eqref{eq:AiAi}. This shows that the collection $\{w(Q')\}_{Q' \in \S(Q)}$ satisfies the Carleson packing condition with the constant $c_n [w]_{A_p}$. As a consequence, this and Lemma \ref{lem:Car-emb} give that
\begin{align*}
\sum_{Q' \in \mathcal{S}(Q)} \inf_{Q'} M(f)^2 w(Q')
&\le \sum_{Q' \in \mathcal{S}(Q)} \bigg(\frac{1}{w(Q')} \int_{Q'} M(f)\, \mathbf{1}_Q w\, dx \bigg)^2 w(Q')
\\%%%%%%%%%%%%%
&\lesssim [w]_{A_p} \|M(f) \mathbf{1}_Q\|_{L^2(w)}^2
=[w]_{A_p} \|M(f) \mathbf{1}_Q\|_{L^2(Q, w)}^2,
\end{align*}
where the above implicit constants are independent of $[w]_{A_p}$ and $Q$. This shows \eqref{eq:QSQ} and completes the proof of \eqref{eq:CF-local}.
\end{proof}
%%%%%%%%%%%%%%%%%%%%%%% END END END PROOF %%%%%%%%%%%%%%%%%%%%%%%%%%

Next, let us see how Lemma \ref{lem:CF-local} implies Theorem \ref{thm:local}.
%%%%%%%%%%%%%%%%%%%%%%%%% PROOF PROOF PROOF %%%%%%%%%%%%%%%%%%%%%%%%
\vspace{0.2cm}
\begin{proof}[\textbf{Proof of Theorem \ref{thm:local}.}]
Let $p>1$ and $r>1$ be chosen later. Define the Rubio de Francia algorithm:
\begin{align*}
\mathcal{R}h=\sum_{k=0}^{\infty} \frac{M^{k}h}{2^k\|M\|^k_{L^{r'}\to L^{r'}}}.
\end{align*}
Then it is obvious that
\begin{align}\label{eq:hRh}
h \le \mathcal{R}h \quad\text{and}\quad \|\mathcal{R}h\|_{L^{r'} (\mathbb{R}^n)} \leq 2 \|h\|_{L^{r'} (\mathbb{R}^n)}.
\end{align}
Moreover, for any nonnegative $h \in L^{r'}(\mathbb{R}^n)$, we have that $\mathcal{R}h \in A_1$ with
\begin{align}\label{eq:Rh-A1}
[\mathcal{R}h]_{A_1} \leq 2 \|M\|_{L^{r'} \to L^{r'}} \leq c_n \ r.
\end{align}

By Riesz theorem and the first inequality in \eqref{eq:hRh}, there exists some nonnegative function $h \in L^{r'}(Q)$ with $\|h\|_{L^{r'}(Q)}=1$ such that
\begin{align}\label{eq:FQ}
\mathscr{F}_Q^{\frac1r} &:= |\{x \in Q:  S_{\alpha,L}(f)(x) > t M(f)(x)\}|^{\frac1r}
\nonumber \\ %%%%%%%%%%%%%%%%
&= |\{x \in Q:  S_{\alpha,L}(f)(x)^2 > t^2 M(f)(x)^2\}|^{\frac1r}
\nonumber \\ %%%%%%%%%%%%%%%%
& \leq \frac{1}{t^2} \bigg\| \bigg(\frac{S_{\alpha,L}(f)}{M(f)}\bigg)^2 \bigg\|_{L^r(Q)}
\leq \frac{1}{t^2} \int_{Q} S_{\alpha,L}(f)^2 \, h \, M(f)^{-2} dx
\nonumber \\ %%%%%%%%%%%%%%%%
&\leq t^{-2} \|S_{\alpha,L}(f)\|_{L^2(Q, w)}^2,
\end{align}
where $w=w_1 w_2^{1-p}$, $w_1= \mathcal{R}h$ and $w_2 = M(f)^{2(p'-1)}$. Recall that Coifmann-Rochberg theorem  \cite[Theorem~3.4]{Jose} asserts that
\begin{align}\label{eq:C-R}
[(M(f))^{\delta}]_{A_1} \leq \frac{c_n}{1-\delta}, \quad\forall \delta \in (0, 1).
\end{align}
In view of \eqref{eq:Rh-A1} and \eqref{eq:C-R}, we see that $w_1, w_2 \in A_1$ provided $p>3$. Then the reverse $A_p$ factorization theorem gives that
$w=w_1 w_2^{1-p} \in A_p$ with
\begin{align}\label{eq:w-Ap-r}
[w]_{A_p} \leq [w_1]_{A_1} [w_2]_{A_1}^{p-1} \leq c_n ~ r.
\end{align}
Thus, gathering \eqref{eq:CF-local}, \eqref{eq:FQ} and \eqref{eq:w-Ap-r},
we obtain
\begin{align*}
\mathscr{F}_Q^{\frac1r}
& \le c_{n} t^{-2}\alpha^{2n} [w]_{A_p} \|M(f)\|_{L^2(Q, w)}^2
 \\ %%%%%%%%
& = c_{n} t^{-2} \alpha^{2n}[w]_{A_p} \|\mathcal{R}h\|_{L^1(Q)}
 \\ %%%%%%%%
&\le c_{n} t^{-2} \alpha^{2n}[w]_{A_p} \|\mathcal{R}h\|_{L^{r'}(Q)} |Q|^{\frac1r}
 \\ %%%%%%%%
&\le c_{n} t^{-2}\alpha^{2n} [w]_{A_p} \|h\|_{L^{r'}(Q)}  |Q|^{\frac1r}
 \\ %%%%%%%%
&\le c_{n} r t^{-2}\alpha^{2n}  |Q|^{\frac1r}.
\end{align*}

Consequently, if $t> \sqrt{c_n e}\,\alpha^{n}$, choosing $r>1$ so that $t^2/e = c_n \alpha^{2n}r$, we have
\begin{align}\label{eq:FQr-1}
\mathscr{F}_Q \le (c_n \alpha^{2n}r t^{-2})^r |Q| = e^{-r}  |Q|
= e^{-\frac{t^2}{c_n e\alpha^{2n}}}  |Q|.
\end{align}
If $0<t \le \sqrt{c_n e}\alpha^{n}$, it is easy to see that
\begin{equation}\label{eq:FQr-2}
\mathscr{F}_Q \le  |Q| \le e \cdot e^{-\frac{t^2}{c_n e\alpha^{2n}}}  |Q|.
\end{equation}
Summing \eqref{eq:FQr-1} and \eqref{eq:FQr-2} up, we deduce that
\begin{equation*}
\mathscr{F}_Q=\mu(\{x \in Q:  S_{\alpha,L}(f)(x) > t M(f)(x)\})
\le c_1 e^{-c_2 t^2/\alpha^{2n}} |Q|,\quad\forall t>0.
\end{equation*}
This proves \eqref{eq:local}.
\end{proof}
%%%%%%%%%%%%%%%%%%%%%%% END END END PROOF %%%%%%%%%%%%%%%%%%%%%%%%%%

%%%%%%%%%%%%%%%%%%%%%%%%% SECTION SECTION SECTION %%%%%%%%%%%%%%%%%%%%%
\subsection{Mixed weak type estimates}
To proceed the proof of Theorem \ref{thm:mixed}, we establish a Coifman-Fefferman inequality.
%%%%%%%%%%%%%%%%%%%%%%%%%% LEMMA LEMMA LEMMA %%%%%%%%%%%%%%%%%%%%%%
\begin{lemma}\label{lem:CF}
For every $0<p<\infty$ and $w \in A_{\infty}$, we have
\begin{align}\label{eq:CF}
\|S_{\alpha, L}f\|_{L^p(w)} &\lesssim \|Mf\|_{L^p(w)},
\end{align}
\end{lemma}
%%%%%%%%%%%%%%%%%%%%%%%%%% LEMMA LEMMA LEMMA %%%%%%%%%%%%%%%%%%%%%%

%%%%%%%%%%%%%%%%%%%%%%%%%%% PROOF PROOF PROOF %%%%%%%%%%%%%%%%%%%%%%
\begin{proof}
Let $w \in A_\infty$. Then, it is well known that for any $\alpha \in (0, 1)$ there exists $\beta \in (0, 1)$ such that for any cube $Q$ and any measurable subset $A\subset Q$
\begin{align*}
|A| \le \alpha |Q| \quad\Longrightarrow\quad w(A) \le \beta w(Q).
\end{align*}
Thus, for the sparsity constant $\eta_j$ of $\mathcal{S}_j$ there exists
$\beta_j \in (0,1)$ such that for $Q \in \mathcal{S}_j$, we have
\begin{equation}\label{eq:w-EQ}
w(E_Q) =w(Q) - w(Q \setminus E_Q) \geq (1-\beta_j)w(Q),
\end{equation}
since $w(Q \setminus E_Q) \leq (1-\eta_j)w(Q)$.
It follows from \eqref{eq:S-sparse} and \eqref{eq:w-EQ} that
\begin{align}\label{eq:p=1}
\int_{\Rn} S_{\alpha,L}(\vec{f})(x)^2 w(x) \ dx
&\lesssim \sum_{j=1}^{3^n} \sum_{Q \in \mathcal{S}_j} \langle |f| \rangle_{Q}^2 w(Q)
\nonumber \\ %%%%%%%%%
&\lesssim \sum_{j=1}^{3^n} \sum_{Q \in \mathcal{S}_j} \Big(\inf_{Q}M(f) \Big)^2 w(E_Q)
\nonumber \\ %%%%%%%%%
&\lesssim \sum_{j=1}^{3^n} \sum_{Q \in \mathcal{S}_j} \int_{E_Q} M(f)(x)^2 w(x) \ dx
\nonumber \\ %%%%%%%%%
&\lesssim \int_{\mathbb{R}^n} M(f)(x)^2 w(x) \ dx.
\end{align}
This shows the inequality \eqref{eq:CF} holds for $p=2$.

To obtain the result \eqref{eq:CF} for every $p \in (0, \infty)$, we apply the $A_{\infty}$ extrapolation theorem from \cite[Corollary 3.15]{CMP11} in the Lebesgue spaces or \cite[Theorem~3.36]{CMM} for the general measure spaces. Let $\mathcal{F}$ be a family of pairs of functions. Suppose that for some $p_0 \in (0, \infty)$ and for every weight $v_0 \in A_{\infty}$,
\begin{align}\label{eq:fg-some}
\|f\|_{L^{p_0}(v_0)} \leq C_1 \|g\|_{L^{p_0}(v_0)}, \quad \forall (f, g) \in \mathcal{F}.
\end{align}
Then for every $p \in (0, \infty)$ and every weigh $v \in A_{\infty}$,
\begin{align}\label{eq:fg-every}
\|f\|_{L^p(v)} \leq C_2 \|g\|_{L^p(v)}, \quad \forall (f, g) \in \mathcal{F}.
\end{align}
From \eqref{eq:p=1}, we see that \eqref{eq:fg-some} holds for the exponent $p_0=2$ and the pair $(S_{\alpha,L}f, Mf)$.   Therefore, \eqref{eq:fg-every} implies that \eqref{eq:CF} holds for the general case $0<p<\infty$.
\end{proof}
%%%%%%%%%%%%%%%%%%%%%%%%%%% END END END PROOF %%%%%%%%%%%%%%%%%%%%%%

\begin{proof}[\textbf{Proof of Theorem \ref{thm:mixed}.}]
In view of Lemma \ref{lem:CF}, we just present the proof for $S_{\alpha,L}$. We use a hybrid of the arguments in \cite{CMP} and \cite{LOPi}. Define
\begin{align*}
\mathcal{R}h(x)=\sum_{j=0}^{\infty} \frac{T^j_u h(x)}{2^j K_0^j},
\end{align*}
where $K_0>0$ will be chosen later and $T_uf(x) := M(fu)(x)/u(x)$ if $u(x) \neq 0$, $T_uf(x)=0$ otherwise. It immediately yields that
\begin{align}\label{eq:R-1}
h \leq \mathcal{R}h \quad
\text{and}\quad T_u(\mathcal{R}h) \leq 2K_0 \mathcal{R}h.
\end{align}
Moreover, follow the same scheme of that in \cite{CMP}, we get for some $r>1$,
\begin{align}\label{eq:R-2}
\mathcal{R}h \cdot uv^{\frac{1}{r'}} \in A_{\infty} \quad\text{and}\quad
\|\mathcal{R}h\|_{L^{r',1}(uv)} \leq 2 \|h\|_{L^{r',1}(uv)}.
\end{align}

Note that
\begin{equation}\label{e:Lpq}
\|f^q\|_{L^{p,\infty}(w)}= \|f\|^q_{L^{pq,\infty}(w)}, \ \ 0<p,q<\infty.
\end{equation}
This implies that
\begin{align*}
& \bigg\| \frac{S_{\alpha,L}(f)}{v} \bigg\|_{L^{1,\infty}(uv)}^{\frac{1}{r}}
= \bigg\| \bigg(\frac{S_{\alpha,L}(f)}{v}\bigg)^{\frac{1}{r}}\bigg\|_{L^{r,\infty}(uv)}
\\ %%%%%%%%%%%%%%%%%%
&=\sup_{\|h\|_{L^{r',1}(uv)}=1}
\bigg|\int_{\Rn} |S_{\alpha,L}(f)(x)|^{\frac{1}{r}} h(x) u(x) v(x)^{\frac{1}{r'}} dx  \bigg|
\\ %%%%%%%%%%%%%%%%%%
&\leq \sup_{\|h\|_{L^{r',1}(uv)}=1}
\int_{\Rn} |S_{\alpha,L}(f)(x)|^{\frac{1}{r}} \mathcal{R}h(x)
u(x) v(x)^{\frac{1}{r'}} dx.
\end{align*}
Invoking Lemma \ref{lem:CF} and H\"{o}lder's inequality, we obtain
\begin{align*}
& \int_{\Rn} |S_{\alpha,L}(f)(x)|^{\frac{1}{r}} \mathcal{R}h(x) u(x) v(x)^{\frac{1}{r'}} dx
\\ %%%%%%%%%%%%%%%%%%
& \lesssim \int_{\Rn} M(f)(x)^{\frac{1}{r}} \mathcal{R}h(x) u(x) v(x)^{\frac{1}{r'}} dx
\\ %%%%%%%%%%%%%%%%%%
& =\int_{\Rn} \bigg(\frac{M(f)(x)}{v(x)}\bigg)^{\frac{1}{r}} \mathcal{R}h(x) u(x) v(x) dx
\\ %%%%%%%%%
& \leq \bigg\|\bigg(\frac{M(f)}{\nu}\bigg)^{\frac{1}{r}} \bigg\|_{L^{r,\infty}(uv)}\|\mathcal{R}h\|_{L^{r',1}(uv)}
\\ %%%%%%%%%
& \leq \bigg\|\frac{M(f)}{v}\bigg\|_{L^{1,\infty}(uv)}^{\frac{1}{r}} \|h\|_{L^{r',1}(uv)},
\end{align*}
where we used \eqref{e:Lpq} and \eqref{eq:R-2} in the last inequality. Here we need to apply the weighted mixed weak type estimates for $M$ proved in Theorems 1.1 in \cite{LOP}. Consequently, collecting the above estimates, we get the desired result
\begin{equation*}
\bigg\| \frac{S_{\alpha,L}(f)}{v} \bigg\|_{L^{1,\infty}(uv)}
\lesssim \bigg\| \frac{M(f)}{v} \bigg\|_{L^{1,\infty}(uv)}
\lesssim \| f \|_{L^1(u)}.
\end{equation*}
The proof is complete.
\end{proof}
%%%%%%%%%%%%%%%%%%%%%%%%%%% END END END PROOF %%%%%%%%%%%%%%%%%%%%%%

%%%%%%%%%%%%%%%%%%%%% SUBSECTION SUBSECTION SUBSECTION %%%%%%%%%%%%%%%%%%
\subsection{Restricted weak type estimates}

%%%%%%%%%%%%%%%%%%%%%%%%%%% PROOF PROOF PROOF %%%%%%%%%%%%%%%%%%%%%%
\begin{proof}[{\bf Proof of Theorem \ref{thm:RW}}]
In view of \eqref{eq:S-sparse}, it suffices to show that
\begin{align}
\|\A_{\S}^2(\mathbf{1}_E)\|_{L^{p, \infty}(w)} \lesssim [w]_{A_p^{\mathcal{R}}}^{p+1} w(E)^{\frac1p}.
\end{align}
Since $\S$ is sparse, for every $Q \in \S$, there exists $E_Q \subset Q$ such that $|E_Q| \simeq |Q|$ and $\{E_Q\}_{Q \in \S}$ is a disjoint family. Note that $Q \subset Q(x, 2\ell(Q)) \subset 3Q$ for any $x \in Q$, where $Q(x, 2\ell(Q))$ denotes the cube centered at $x$ and with sidelength $2\ell(Q)$. Hence, for all $Q \in \S$ and for all $x \in Q$,
\begin{align}\label{eq:EQQ}
\frac{w(Q(x, 2\ell(Q)))}{w(E_Q)} \simeq \frac{w(Q)}{w(E_Q)}
\le [w]_{A_p^{\mathcal{R}}}^p \bigg(\frac{|Q|}{|E_Q|}\bigg)^p
\lesssim [w]_{A_p^{\mathcal{R}}}^p.
\end{align}
By duality, one has
\begin{align}\label{eq:AE}
\|\A_{\S}^2(\mathbf{1}_E)\|_{L^{p, \infty}(w)}^2
=\|\A_{\S}^2(\mathbf{1}_E)^2\|_{L^{p/2, \infty}(w)}
&=\sup_{\substack{0 \le h \in L^{(p/2)', 1}(w) \\ \|h\|_{L^{(p/2)',1}(w)}=1}}
\int_{\Rn} \A_{\S}^2(\mathbf{1}_E)^2 hw\, dx.
\end{align}
Fix such $h$ above.  Then, using \eqref{eq:EQQ}, the disjointness of $\{E_Q\}_{Q \in \S}$, H\"{o}lder's inequality and \eqref{eq:ME}, we conclude that
\begin{align}\label{eq:AEE}
&\int_{\Rn} \A_{\S}^2(\mathbf{1}_E)^2 hw\, dx
=\sum_{Q \in \S} \langle \mathbf{1}_E \rangle_Q^2 \int_{Q} hw\, dx
\nonumber\\%%%%%%%%%%%%%%
&\le \sum_{Q \in \S} \langle \mathbf{1}_E \rangle_Q^2 \bigg(
\fint_{Q(x, 2\ell(Q))} h\, dw\bigg) w(E_Q) \bigg(\frac{w(Q(x, 2\ell(Q)))}{w(E_Q)}\bigg)
\nonumber\\%%%%%%%%%%%%%%
&\lesssim  [w]_{A_p^{\mathcal{R}}}^p \sum_{Q \in \S} \Big(\inf_Q M\mathbf{1}_E\Big)^2 \Big(\inf_Q M_w^c h\Big) w(E_Q)
\nonumber\\%%%%%%%%%%%%%%
&\le  [w]_{A_p^{\mathcal{R}}}^p \int_{\Rn} M\mathbf{1}_E(x)^2 M_w^c h(x) w\, dx
\nonumber\\%%%%%%%%%%%%%%
&\le  [w]_{A_p^{\mathcal{R}}}^p \|(M\mathbf{1}_E)^2\|_{L^{p/2,\infty}(w)} \|M_w^c h\|_{L^{(p/2)', 1}(w)}
\nonumber\\%%%%%%%%%%%%%%
&\lesssim  [w]_{A_p^{\mathcal{R}}}^{p+2} w(E)^{\frac2p}.
\end{align}
Therefore,  \eqref{eq:AE} and \eqref{eq:AEE} immediately imply \eqref{eq:SLE}.
\end{proof}

\subsection{Endpoint estimates for commutators}
Recall that the sharp maximal function of $f$ is defined by
\[
M_{\delta}^{\sharp}(f)(x):=\sup_{x\in Q} \inf_{c \in \R} \bigg(\fint_{Q}|f^{\delta}-c|dx\bigg)^{\frac{1}{\delta}}.
\]
If we write $Q_{t, L} := t^m Le^{-t^m L}$, then
\begin{align*}
C_b(S_{\alpha, L})f(x) &= \bigg(\iint_{\Gamma_{\alpha}(x)} |Q_{t, L} \big((b(x)-b(\cdot))f(\cdot) \big)(y)|^2 \frac{dydt}{t^{n+1}} \bigg)^{\frac12}.
\end{align*}

%%%%%%%%%%%%%%%%%%%%%%%%%% LEMMA LEMMA LEMMA %%%%%%%%%%%%%%%%%%%%%%
\begin{lemma}\label{lem:MMf}
For any $0<\delta<1$,
\begin{align}
M_{\delta}^{\#}(\widetilde{S}_{\alpha, L} f)(x_{0}) \lesssim  \alpha^{2n} Mf(x_{0}),  \quad \forall  x_{0} \in \Rn.
\end{align}
\end{lemma}
%%%%%%%%%%%%%%%%%%%%%%%%%% LEMMA LEMMA LEMMA %%%%%%%%%%%%%%%%%%%%%%
\begin{proof}
For any cube $Q \ni x_{0}$.  The lemma will be proved if we can show that

\[
\left(\fint_{Q}|\widetilde{S}_{\alpha,L}(f)^{2}(x)-c_{Q}|^{\delta}dx\right)^{\frac{1}{\delta}}\lesssim \alpha^{2n}Mf(x_{0})^{2},
\]
where $c_{Q}$ is a constant which will be determined later.

Denote $T(Q)=Q\times(0,\ell(Q))$. We write
\[
\widetilde{S}_{\alpha,L}(f)^{2}(x) =E(f)(x)+F(f)(x),
\]
where
\begin{align*}
E(f)(x) &:=\iint_{T(2Q)}\Phi\Big(\frac{|x-y|}{\alpha t}\Big)|Q_{t, L}f(y)|^{2}\frac{dydt}{t^{n+1}},
\\
F(f)(x) &:= \iint_{\R_{+}^{n+1}\backslash T(2Q)}\Phi\Big(\frac{|x-y|}{\alpha t}\Big)|Q_{t, L}f(y)|^{2}\frac{dydt}{t^{n+1}}.
\end{align*}
Let us choose $c_{Q}=F(f)(x_{Q})$ where $x_{Q}$ is the center of $Q$. Then
\begin{align*}
\bigg(&\fint_{Q}|\widetilde{S}_{\alpha,\psi}(f)^{2}-c_{Q}|^{\delta} dx\bigg)^{\frac{1}{\delta}}
=\left(\fint_{Q}|E(f)(x)+F(f)(x)-c_{Q}|^{\delta} dx\right)^{\frac{1}{\delta}}
\\
&\qquad \lesssim\left(\fint_{Q}|E(f)(x)|^{\delta} dx\right)^{\frac{1}{\delta}}+\left(\fint_{Q}|F(f)(x)-F(f)(x_{Q})|^{\delta}dx \right)^{\frac{1}{\delta}}
=:I+II
\end{align*}
We estimate each term separately. For the first term $I$, we set $f=f_0+f^\infty$, where $f_{0}=f\chi_{Q^{*}}, f^\infty=f\chi_{(Q^{*})^{c}}$ and $Q^{*}=8Q$.  Then we have
\begin{equation}\label{E0}
E(f)(x)\lesssim E(f_{0})(x)+E(f^\infty)(x).
\end{equation}
Therefore,
\[
\left(\fint_{Q}|E(f)(x)|^{\delta} dx\right)^{\frac{1}{\delta}}
\lesssim\left(\fint_{Q}|E(f_{0})(x)|^{\delta} dx\right)^{\frac{1}{\delta}}
+\left(\fint_{Q}|E(f^\infty (x)|^{\delta}dx\right)^{\frac{1}{\delta}}.
\]

It was proved in \cite[p. 884]{BD}, that
$\|\widetilde{S}_{\alpha,L}(f)\|_{L^{1,\infty}}\lesssim  \alpha^{n}\|S_{1,L}(f)\|_{L^{1,\infty}}$.
Then, by \eqref{eq:S11} and Kolmogorov inequality we have
\begin{align}\label{E1}
\bigg(\fint_{Q} & |E(f_{0})(x)|^{\delta} dx\bigg)^{\frac{1}{\delta}}
\leq\left(\fint_{Q}|\widetilde{S}_{\alpha,L}(f_0)|^{2\delta} dx \right)^{\frac{2}{2\delta}}
\nonumber\\
& \lesssim\|\widetilde{S}_{\alpha,L}(f_0)\|_{L^{1,\infty}(Q,\frac{dx}{|Q|})}^{2}
\lesssim \alpha^{2n}\Big(\fint_{Q^{*}}|f_{0}(x)|dx\Big)^{2}.
\end{align}
On the other hand,
\[
\begin{aligned}
\left(\fint_{Q}|E(f^\infty)(x)|^{\delta}dx\right)^{\frac{1}{\delta}}&\lesssim\frac 1{|Q|}\int_{\mathbb{R}^{n}}\int_{T(2Q)}\Phi\Big(\frac{x-y}{\alpha t}\Big)\Big|Q_{t, L}f(y)\Big|^{2}\frac{dydt}{t^{n+1}}dx\\
&=\frac{\alpha^{2n}}{|Q|}\int_{T(2Q)}|Q_{t, L}(f^\infty)(y)|^{2}\frac{dydt}{t},
\end{aligned}
\]
since
$\int_{\mathbb{R}^{n}}\Phi\Big(\frac{x-y}{\alpha t}\Big)dx\leq c_{n}(\alpha t)^{n}$.

For any $k\in N_+$, $p_{k,t}(x,y)$  denote the kernel of operator $(tL)^k e^{-tL}$. Note that condition (A2) implies that for any  $\delta_0>0$,  there exist $C, c>0$ such that
%$$|p_{k,t}(x,y)|\leq \frac{C}{t^{n/ m}}\exp\bigg(-\frac{|x-y|^{m/(m-1)}}{c \, t^{1/(m-1)}}\bigg)$$
%and
\begin{align}\label{kernelestimate}
|p_{k,t}(x,y)|\leq \frac{C}{t^{n/ m}} \bigg(\frac{t^{1/ m}}{t^{1/ m}+|x-y|}\bigg)^{n+\delta_0},\; \text{for all} \; x, y\in \mathbb{R}^n.
\end{align}

Thus, \eqref{kernelestimate} implies that
\[
\begin{aligned}
&\bigg(\int_{2Q}|Q_{t, L}(f^\infty)(y)|^{2}dy\bigg)^{1/ 2}\\
&\lesssim \sum_{j\geq 3}\bigg \{\int_{2Q} \bigg[\int_{2^{j+1}Q\setminus 2^j Q}\frac{1}{t^n}\bigg(\frac{t}{t+|y-z|}\bigg)^{n+\delta_0}|f(z)|dz\bigg]^2dy\bigg\}^{1/2}\\
&\lesssim \sum_{j\geq 3}\bigg \{\int_{2Q} \bigg[\int_{2^{j+1}Q\setminus 2^j Q}\frac{1}{t^n}\bigg(\frac{t}{2^j\ell(Q)}\bigg)^{n+\delta_0}|f(z)|dz\bigg]^2dy\bigg\}^{1/2}\\
&\lesssim  \bigg(\frac{t}{\ell(Q)}\bigg)^{\delta_0} |2Q|^{1/ 2} \sum_{j\geq 3} \frac{1}{2^{j\delta_0}} \bigg(\fint_{2^{j}Q}|f(z)|dz\bigg).
\end{aligned}
\]

Then one has
\begin{align*}
& \left(\fint_{Q}|E(f^\infty)(x)|^{\delta}dx\right)^{\frac{1}{\delta}}
\\
& \lesssim \bigg[\sum_{l=0}^{\infty}\frac 1{2^{l\delta_0}}\bigg(\fint_{2^{l}Q}|f_{j}| dx \bigg)\bigg]^{2}\frac{\alpha^{2n}}{|Q|}\int_{T(2Q)}|2Q|(t/\ell(Q))^{2\delta_0}\frac{dydt}{t}
\\
&\lesssim \alpha^{2n}\bigg[\sum_{l=0}^{\infty}\frac 1{2^{l\delta_0}}\bigg(\fint_{2^{l}Q}|f_{j}| dx \bigg)\bigg]^{2}
\lesssim \alpha^{2n}\sum_{l=0}^{\infty}\frac 1{2^{l\delta_0}}\bigg(\fint_{2^{l}Q}|f_j| dx\bigg)^{2},
\end{align*}
where in the last inequality we used H\"older's inequality.

Therefore, we obtain that
\begin{equation}\label{E2}
\left(\fint_{Q}|E(f^\infty)(x)|^{\delta}dx\right)^{\frac{1}{\delta}}
\lesssim \alpha^{2n}\sum_{l=0}^{\infty}\frac 1{2^{l\delta_0}}\bigg(\fint_{2^{l}Q}|f| dx\bigg)^{2}.
\end{equation}
Gathering \eqref{E0}, \eqref{E1} and \eqref{E2}, we deduce that
\[
I\lesssim\alpha^{2n}M(f)(x_{0}).
\]
To estimate the second term $II$, we shall use the following estimate \cite[Eq (35)]{BD}:
\begin{equation}
|F(f)(x)-F(f)(x_{Q})|\lesssim \alpha^{2n}\sum_{l=0}^{\infty}\frac 1{2^{l\delta}}\bigg(\fint_{2^{l}Q}|f| dx\bigg)^{2},
\end{equation}
for some $\delta>0$ and  all $x\in Q$, where $x_Q$ is the center of $Q$.  As a consequence, we have
\[
II=\left(\fint_{Q}|F(f)(x)-F(f)(x_{Q})|^{\delta} dx \right)^{\frac{1}{\delta}}\lesssim\alpha^{2n}Mf(x_{0}).
\]
This finish the proof. 
\end{proof}

%%%%%%%%%%%%%%%%%%%%%%%%%% LEMMA LEMMA LEMMA %%%%%%%%%%%%%%%%%%%%%%
\begin{lemma}\label{lem:MSL}
For any $0<\delta<\varepsilon<1$ and for any $b \in \BMO$,
\begin{align}
M_{\delta}^{\#}(C_b(\widetilde{S}_{\alpha, L})f)(x) \lesssim
\|b\|_{\BMO} \big(M_{L\log L}f(x) +M_{\varepsilon}(\widetilde{S}_{\alpha, L}f)(x) \big).
\end{align}
\end{lemma}
%%%%%%%%%%%%%%%%%%%%%%%%%% LEMMA LEMMA LEMMA %%%%%%%%%%%%%%%%%%%%%%

%%%%%%%%%%%%%%%%%%%%%%%%%% PROOF PROOF PROOF %%%%%%%%%%%%%%%%%%%%%%%
\begin{proof}
Let $x\in \Rn$, and let $Q$ be any arbitrary cube containing $x$. It suffices to show that there exists $c_Q$ such that
\begin{align}\label{eq:Cbb}
\mathscr{A}&:=\bigg(\fint_Q |C_b(\widetilde{S}_{\alpha, L})f(z)-c_Q|^{\delta} dz\bigg)^{\frac{1}{\delta}}
\lesssim \|b\|_{\BMO} \big(M_{L\log L}f(x) +M_{\varepsilon}(\widetilde{S}_{\alpha, L}f)(x) \big).
\end{align}
Split $f=f_1+f_2$, where $f_1=f {\bf 1}_{8Q}$. Then, we have
\begin{align}\label{eq:AAA}
\mathscr{A} & \lesssim \bigg(\fint_{Q} |(b(z)-b_Q)\widetilde{S}_{\alpha, L}(f)(z)|^{\delta} dz\bigg)^{\frac{1}{\delta}}
\nonumber\\
&\qquad+ \bigg(\fint_Q |\widetilde{S}_{\alpha, L}((b-b_Q)f_{1})(z)|^{\delta} dz\bigg)^{\frac{1}{\delta}}
\nonumber\\
&\qquad + \bigg(\fint_Q |\widetilde{S}_{\alpha, L}((b-b_Q)f_2)(z)-c_Q|^{\delta} dz\bigg)^{\frac{1}{\delta}}
\nonumber\\
&:=\mathscr{A}_1 + \mathscr{A}_2 + \mathscr{A}_3.
\end{align}
To bound $\mathscr{A}_1$, we choose $r \in (1, \varepsilon/\delta)$. The H\"{o}lder's inequality gives that
\begin{align}\label{eq:A1}
\mathscr{A}_1
& \leq \bigg(\fint_Q |b(z)-b_Q|^{\delta r'} dz\bigg)^{\frac{1}{\delta r'}}
\bigg(\fint_Q |\widetilde{S}_{\alpha, L}f(z)|^{\delta r} dz\bigg)^{\frac{1}{\delta r}}
\nonumber\\
&\lesssim \|b\|_{\BMO} M_{\delta r}(\widetilde{S}_{\alpha, L}f)(x)
\le \|b\|_{\BMO} M_{\varepsilon}(\widetilde{S}_{\alpha, L}f)(x).
\end{align}
Since $\widetilde{S}_{\alpha, L}: L^{1}(\Rn)\rightarrow L^{1,\infty}(\Rn)$ and $0<\delta <1$, there holds
\begin{align}\label{eq:A2}
\mathscr{A}_2 & \lesssim \|\widetilde{S}_{\alpha, L}((b-b_Q)f_1)\|_{L^{1,\infty}(Q, \frac{dx}{|Q|})}
\lesssim \fint_Q |(b-b_Q)f_1|dz
\nonumber\\
&\lesssim \|b-b_Q\|_{\exp L, Q} \|f\|_{L\log L, Q} \lesssim \|b\|_{\BMO} M_{L\log L}(f)(x).
\end{align}
For the last term, we take $c_Q=\widetilde{S}_{\alpha, L}((b-b_Q)f_2)(z_Q)$, where $z_Q$ is the center of $B$. We have
\begin{align}\label{eq:A3}
\mathscr{A}_3 \le \fint_Q |\widetilde{S}_{\alpha, L}((b-b_Q)f_2)(z)-c_Q|\, dz
=:\fint_Q J_Q(z)\, dz \le \bigg(\fint_Q J_Q(z)^2\, dz\bigg)^{\frac12}.
\end{align}
For any cube $Q\subset \Rn$, set $T_Q=Q \times (0, \ell(Q))$. Thus, for any $z \in Q$,
\begin{align}\label{eq:JQJQ}
J_Q(z)^2 & \leq |\widetilde{S}_{\alpha, L}((b-b_Q)f_2)(z)^2 - \widetilde{S}_{\alpha, L}((b-b_Q)f_2)(z_Q)^2|
\nonumber\\
&\leq \iint_{T(2Q)} \Phi \Big(\frac{z-y}{\alpha t}\Big) |Q_{t, L}((b-b_Q)f_2)(y)|^2 \frac{dydt}{t^{n+1}}
\nonumber\\
&\qquad+ \iint_{T(2Q)} \Phi \Big(\frac{z'-y}{\alpha t}\Big) |Q_{t, L}((b-b_Q)f_2)(y)|^2 \frac{dydt}{t^{n+1}}
\nonumber\\
&\qquad + \iint_{\R^{n+1} \setminus T(2Q)} \Big|\Phi \Big(\frac{z-y}{\alpha t}\Big) - \Phi \Big(\frac{z'-y}{\alpha t}\Big)\Big|
|Q_{t^m}((b-b_Q)f_2)(y)|^2 \frac{dydt}{t^{n+1}}
\nonumber\\
&=: J_{Q, 1}(z)+J_{Q, 2}(z)+ J_{Q, 3}(z).
\end{align}
In order to estimate $J_{Q, 1}(z)$, we note that
\begin{align}\label{eq:QPhi}
\fint_Q \Phi \Big(\frac{z-y}{\alpha t}\Big) dz
\leq \frac{1}{|Q|} \int_{\Rn} \Phi \Big(\frac{z-y}{\alpha t}\Big) dz
\lesssim \frac{(\alpha t)^{n}}{|Q|}.
\end{align}
Furthermore, the kernel estimate \eqref{kernelestimate} gives that
\begin{align}\label{eq:QtLbb}
\bigg(\int_{2Q} & |Q_{t, L}((b-b_Q)f_2)(y)|^2\bigg)^{\frac12}
\nonumber\\
&\lesssim \sum_{j\geq 3} \bigg\{\int_{2Q} \bigg[\int_{2^{j+1}Q \setminus 2^j Q} \frac{1}{t^{n}}\bigg(\frac{t}{t+|z-y|}\bigg)^{n+\delta_0}|b-b_Q| |f| dz\bigg]^{2}dy\bigg\}^{\frac12}
\nonumber\\
&\lesssim \sum_{j\geq 3}\bigg\{\int_{2Q}\bigg[\int_{2^{j+1}Q \setminus 2^jQ} \frac{1}{t^n} \bigg(\frac{t}{2^{j}\ell(Q)}\bigg)^{n+\delta_0} |b-b_Q| |f| dz\bigg]^2 dy \bigg\}^{\frac12}
\nonumber\\
&\lesssim \sum_{j\geq 3} \bigg(\frac{t}{2^{j}\ell(Q)}\bigg)^{\delta_0} |Q|^{\frac{1}{2}}\fint_{2^{j+1}Q}|b-b_Q| |f| dz
\nonumber\\
&\lesssim \bigg(\frac{t}{\ell(Q)}\bigg)^{\delta_0} |Q|^{\frac{1}{2}} \sum_{j\geq 0} 2^{-j\delta_0} \|b-b_Q\|_{\exp L, 2^{j+1}Q} \|f\|_{L\log L, 2^{j+1}Q}
\nonumber\\
&\lesssim \bigg(\frac{t}{\ell(Q)}\bigg)^{\delta_0} |Q|^{\frac{1}{2}} \sum_{j\geq 0} 2^{-j\delta_0} j \|b\|_{\BMO} M_{L\log L}f(x)
\nonumber\\
&\lesssim \bigg(\frac{t}{\ell(Q)}\bigg)^{\delta_0} |Q|^{\frac{1}{2}} \|b\|_{\BMO} M_{L\log L}f(x).
\end{align}
Then, gathering \eqref{eq:QPhi} and \eqref{eq:QtLbb}, we obtain
\begin{align}\label{eq:JQ1}
\fint_Q J_{Q, 1}(z)dz
&\leq \iint_{T(2Q)} \bigg(\fint_{Q}\Phi\Big(\frac{z-y}{\alpha t}\Big)dz\bigg) |Q_{t, L}((b-b_Q)f_2)(y)|^{2}\frac{dydt}{t^{n+1}}
\nonumber\\
&\lesssim \frac{1}{|Q|} \int_{0}^{2\ell(Q)} \int_{2Q}|Q_{t, L}((b-b_Q)f_2)(y)|^2 dy \frac{dt}{t}
\nonumber\\
&\lesssim \int_{0}^{2\ell(Q)} \bigg(\frac{t}{\ell(Q)}\bigg)^{2\delta_0} \frac{dt}{t} \, \|b\|_{\BMO}^2 M_{L\log L}f(x)^2
\nonumber\\
&\lesssim \|b\|_{\BMO}^2 M_{L\log L}f(x)^2.
\end{align}
Similarly, one has
\begin{align}\label{eq:JQ2}
\fint_Q J_{Q, 2}(z)dz \lesssim \|b\|_{\BMO} M_{L\log L}f(x).
\end{align}
To control $J_{Q, 3}$, invoking \cite[eq. (35)]{BD}, we have
\begin{align}\label{eq:JQ3}
J_{Q, 3}(z) &\lesssim \sum_{j\geq 0} 2^{-j\delta_{0}} \bigg(\fint_{2^j Q} |b-b_Q| |f| dz \bigg)^2
\nonumber\\
&\lesssim \sum_{j \geq 0} 2^{-j\delta_0} \|b-b_Q\|^{2}_{\exp L, 2^j Q} \|f\|_{L\log L,2^j Q}
\nonumber\\
&\lesssim \sum_{j\geq 0} 2^{-j\delta_0} j \|b\|_{\BMO} M_{L\log L}f(x)
\lesssim \|b\|_{\BMO} M_{L\log L}f(x).
\end{align}
Combining \eqref{eq:A3}, \eqref{eq:JQJQ}, \eqref{eq:JQ1}, \eqref{eq:JQ2} and \eqref{eq:JQ3}, we conclude that
\begin{equation}\label{eq:A33}
\mathscr{A}_3 \lesssim \|b\|_{\BMO} M_{L\log L}f(x).
\end{equation}
Therefore, \eqref{eq:Cbb} immediately follows from \eqref{eq:AAA}, \eqref{eq:A1}, \eqref{eq:A2} and \eqref{eq:A33}.
\end{proof}
%%%%%%%%%%%%%%%%%%%%%%%%%% END END END PROOF %%%%%%%%%%%%%%%%%%%%%%%

\begin{lemma}\label{lem:WEC}
For any $w \in A_{\infty}$ and $b\in \BMO$,
\begin{equation}\label{eq:SLML}
\begin{split}
\sup_{t>0} \Phi(1/t)^{-1} & w(\{x\in \Rn: |C_b(S_{\alpha,L})f(x)|>t\})
\\
&\lesssim \sup_{t>0} \Phi(1/t)^{-1} w(\{x\in \Rn: M_{L\log L}f(x)>t\}),
\end{split}
\end{equation}
for all $f \in L^{\infty}_c(\Rn)$.
\end{lemma}

\begin{proof}
Recall that the weak type Fefferman-Stein inequality:
\begin{align}\label{WF-S}
\sup_{\lambda>0} \varphi(\lambda)\omega(\{x\in \Rn:M_{\delta}f(x)>\lambda\})
\leq \sup_{\lambda>0}\varphi(\lambda)\omega(\{x\in \Rn:M^{\sharp}_{\delta}f(x)>\lambda\})
\end{align}
for all function $f$ for which the left-hand side is finite, where $\varphi:(0,\infty)\rightarrow(0,\infty)$ is doubling. We may assume that the right-hand side of \eqref{eq:SLML} is finite since otherwise there is nothing to be proved. Now by the Lebesgue diffentiation theorem we have
\begin{align*}
\mathscr{B}:=&\sup_{t>0} \Phi(1/t)^{-1} w(\{x\in \Rn: |C_b(S_{\alpha,L})f(x)|>t\})
\\
=&\sup_{t>0} \Phi(1/t)^{-1} w(\{x\in \Rn: |C_b(\widetilde{S}_{\alpha,L})f(x)|>t\})
\\
\leq & \sup_{t>0} \Phi(1/t)^{-1} w(\{x\in \Rn: M_{\delta}(C_b(\widetilde{S}_{\alpha,L}))f(x)|>t\}).
\end{align*}
Then Lemma \ref{lem:MMf}, Lemma \ref{lem:MSL} and \eqref{WF-S} give that
\begin{align*}
\mathscr{B} & \lesssim \sup_{t>0} \Phi(1/t)^{-1} w(\{x\in \Rn: M_{\delta}^{\#}(C_b(\widetilde{S}_{\alpha, L})f)(x) >t\})
\\
&\lesssim \sup_{t>0} \Phi(1/t)^{-1} w(\{x\in \Rn: M_{L\log L}f(x) +M_{\varepsilon}(\widetilde{S}_{\alpha, L}f)(x)>c_0t\})
\\
&\lesssim \sup_{t>0} \Phi(1/t)^{-1} w(\{x\in \Rn:M_{L\log L}f(x)>t\})
\\
&\quad + \sup_{t>0} \Phi(1/t)^{-1} w(\{x\in \Rn:M_{\varepsilon}(\widetilde{S}_{\alpha, L}f)(x)>t\})
\\
&\lesssim \sup_{t>0} \Phi(1/t)^{-1} w(\{x\in \Rn:M_{L\log L}f(x)>t\})\\
&\quad +\sup_{t>0} \Phi(1/t)^{-1} w(\{x\in \Rn:M^{\sharp}_{\varepsilon}(\widetilde{S}_{\alpha, L}f)(x)>t\})
\\
&\lesssim \sup_{t>0} \Phi(1/t)^{-1} w(\{x\in \Rn:M_{L\log L}f(x)>t\})
\\
&\quad + \sup_{t>0} \Phi(1/t)^{-1} w(\{x\in \Rn:M(f)(x)>t\})
\\
&\lesssim \sup_{t>0} \Phi(1/t)^{-1} w(\{x\in \Rn:M_{L\log L}f(x)>t\}).
\end{align*}
The proof is complete.
\end{proof}

\begin{proof}[\textbf{Proof of Theorem \ref{thm:SbA1}.}]
Let $w \in A_1$. By homogeneity, it is enough to prove
\begin{align}\label{eq:Cbt=1}
w(\{x\in \Rn: C_b(S_{\alpha,L})f(x)>1\}) \lesssim \int_{\Rn} \Phi(|f(x)|) w(x) dx.
\end{align}
Let us recall a result from \cite[Lemma~2.11]{CY} for $m=1$. For any $w \in A_1$,
\begin{align}\label{eq:MLL}
w(\{x\in \Rn: M_{L\log L}f(x)>t\}) \lesssim \int_{\Rn} \Phi \bigg(\frac{|f(x)|}{t}\bigg) w(x) dx, \quad\forall t>0.
\end{align}
Since $\Phi$ is submultiplicative, Lemma \ref{lem:WEC} and \eqref{eq:MLL} imply that
\begin{align*}
&w(\{x\in \Rn: C_b(S_{\alpha,L})f(x)>1\})
\\
&\qquad \lesssim \sup_{t>0} \Phi(1/t)^{-1} w(\{x\in \Rn: |C_b(S_{\alpha,L})f(x)|>t\})
\\
&\qquad \lesssim \sup_{t>0} \Phi(1/t)^{-1} w(\{x\in \Rn:M_{L\log L}f(x)>t\})
\\
&\qquad \lesssim \sup_{t>0} \Phi(1/t)^{-1} \int_{\Rn} \Phi \bigg(\frac{|f(x)|}{t}\bigg) w(x) dx
\\
&\qquad \leq \sup_{t>0} \Phi(1/t)^{-1} \int_{\Rn} \Phi (|f(x)|) \Phi(1/t) w(x)dx
\\
&\qquad\le \int_{\Rn} \Phi(|f(x)|) w(x) dx.
\end{align*}
This shows \eqref{eq:Cbt=1} and hence Theorem \ref{thm:SbA1}.
\end{proof}

%%%%%%%%%%%%%%%%%%% BIBLIGRAPHY BIBLIGRAPHY BIBLIGRAPHY %%%%%%%%%%%%%%%%%%
%%%%%%%%%%%%%%%%%%% BIBLIGRAPHY BIBLIGRAPHY BIBLIGRAPHY %%%%%%%%%%%%%%%%%%

\end{document}